\documentclass[a4paper,12pt]{article}
\usepackage[latin1]{inputenc}
\usepackage{amsthm}
\usepackage{amsfonts}
\usepackage{mathrsfs}
\usepackage[dvips]{graphicx}
\usepackage{fancyhdr}
\usepackage{latexsym}
\usepackage{amsmath}
\usepackage{color}
\usepackage{stmaryrd}

\def\red#1{\textcolor{black}{#1}}
\def\blue#1{\textcolor{black}{#1}}

\newtheorem{lem}{Lemma}
\newtheorem{thm}{Theorem}
\newtheorem{defn}{Definition}
\newtheorem{prop}{Proposition}
\newtheorem{cor}{Corollary}
\newtheorem{oss}{Remark}

\newcommand{\uvec}{\boldsymbol{u}}

\newcommand{\vvec}{\boldsymbol{v}}
\newcommand{\xivec}{\boldsymbol\xi}
\newcommand{\pvec}{\boldsymbol{p}}
\newcommand{\hvec}{\boldsymbol{h}}

\newcommand{\nvec}{\boldsymbol{n}}
\newcommand{\wvec}{\boldsymbol{w}}
\newcommand{\ptil}{\widetilde{\boldsymbol{p}}}
\newcommand{\qtil}{\widetilde{q}}
\newcommand{\zvec}{\boldsymbol{z}}

\allowdisplaybreaks[4]

\numberwithin{equation}{section}

\begin{document}

\title{Optimal distributed control of two-dimensional \\ nonlocal Cahn--Hilliard--Navier--Stokes systems \\ with
degenerate mobility and singular potential}

\author{
{Sergio Frigeri \thanks{{\color{black}  Dipartimento di Matematica ``N. Tartaglia'', Universit\`{a} Cattolica
del Sacro Cuore (Brescia), Via Musei 41, Brescia I-25121, Italy.
\textit{sergiopietro.frigeri@unicatt.it, sergio.frigeri.sf@gmail.com} }}}
\and
{Maurizio Grasselli\thanks{Dipartimento di Matematica,
Politecnico di Milano, Via E.~Bonardi 9,
Milano I-20133, Italy,
\textit{mau\-ri\-zio.grasselli@polimi.it}}}
\and
{J\"{u}rgen Sprekels\thanks{{\color{black}
Weierstrass Institute for Applied Analysis and Stochastics, Mohrenstrasse 39, D-10117 Berlin,
Germany\red{, and Department of Mathematics, Humboldt-Universit\"at zu Berlin, Unter den
Linden 6, D-10099 Berlin, Germany},
\textit{juergen.sprekels@wias-berlin.de}}}}}

\maketitle

\begin{abstract}\noindent
\red{In this paper,}  we consider a two-dimensional diffuse interface model for the phase separation of an
incompressible and isothermal binary fluid mixture with matched densities. This model consists of
the Navier--Stokes equations, nonlinearly coupled with a convective nonlocal
Cahn--Hilliard equation. The system rules the evolution of the (\blue{volume-}averaged) velocity $\uvec$ of the mixture and
the (relative) concentration difference $\varphi$ of the two phases.
The \red{aim} of this \blue{work} is to study an optimal
control problem for such a system, the control being a time-dependent external force acting on the fluid.
We first prove the existence of an optimal control for a given tracking type cost functional. Then we
study the differentiability \red{properties} of the control-to-state map $\vvec \mapsto [\uvec,\varphi]$, and we establish
first-order necessary optimality conditions. These results generalize the ones obtained by the first
and the third authors jointly with E.~Rocca in \red{\cite{FRS}. There the authors assumed} a
constant mobility and a regular potential with polynomially controlled growth. Here, we analyze the
physically more relevant case of a degenerate mobility and a singular (e.g., logarithmic) potential.
This is made possible by the existence of a unique strong solution \red{which was recently}
 proved by the authors and C.\,G.~Gal \red{in \cite{FGGS}}.

\end{abstract}

\section{Introduction}\label{intro}
\setcounter{equation}{0}

A well-known diffuse interface model for incompressible and isothermal
binary fluids is the so-called Cahn--Hilliard--Navier--Stokes system (see, for instance, \cite{AMW,GPV,HMR}).
\red{It consists of} the nonlinear coupling of the Navier--Stokes equations for the \blue{volume-averaged} velocity $\uvec$
with the convective Cahn--Hilliard equation for the (relative) concentration difference $\varphi$
of the two \blue{fluids}. More precisely, assuming matched densities equal to \red{unity}, we have \red{to deal with
the system of partial differential equations}
\begin{align}
\label{sy01}
&\uvec_t-2\,\mbox{div}\left(\nu(\varphi)D\uvec\right)+(\uvec\cdot\nabla)\uvec+\nabla\pi=\mu\nabla\varphi+\vvec, \\
\label{sy02}
& \varphi_t+\uvec\cdot\nabla\varphi=\mbox{div}(m(\varphi)\nabla\mu), \\
\label{sy03}
& \mbox{div}(\uvec)=0,
\end{align}
in $Q:=\Omega\times(0,T)$, where $\Omega\subset\mathbb{R}^d$, $d=2,3$, is a bounded and smooth domain, $T>0$
is a prescribed final time, and $D$ denotes the symmetric gradient defined by $D\uvec:=\big(\nabla \uvec+\nabla^T\uvec
\big)/2$. Here, the viscosity $\nu(\cdot)$ is strictly positive, $\pi$ \blue{stands for} the pressure, $\vvec$
\blue{is} a given external force density,
$m(\cdot)$ is the mobility, and $\mu$ represents the so-called chemical potential.
Within the phenomenological framework devised in \cite{CH}, $\mu$ is the functional derivative of the
\red{local} Ginzburg--Landau type functional
\begin{equation}\label{GL}
{\cal G}(\varphi) = \int_\Omega \left(\frac{\vert \nabla \varphi\vert^2}{2} + W(\varphi)\right)dx,
\end{equation}
where $W$ is
a given double-well potential. Here, and in the following, all \red{of the relevant physical} constants have been set equal to \red{unity}, for the sake of simplicity.

On the other hand, a physically more rigorous approach shows that $\mu$ is the
functional derivative of a {\em nonlocal} functional of the following form (see \cite{BELM,GL0,GL1,GL2}, cf. also \cite{GGG}
for a detailed discussion):
$$
{\cal F}(\varphi) = -\frac{1}{2}\int_\Omega\int_\Omega K(x-y)\varphi(x)\varphi(y)dxdy + \int_\Omega F(\varphi) dx\,.
$$
\red{Here,} $K:\mathbb{R}^d\to \mathbb{R}$ is a sufficiently smooth interaction kernel such that $K(x)=K(-x)$, and $F$
is a convex potential (usually of logarithmic type).

In this contribution, we address \red{an optimal control problem for} the following nonlocal Cahn--Hilliard--Navier--Stokes system:
\begin{align}
&\uvec_t-2\,\mbox{div}\left(\nu(\varphi)D\uvec\right)+(\uvec\cdot\nabla)\uvec+\nabla\pi=\mu\nabla\varphi+\vvec,\label{sy3}\\
& \varphi_t+\uvec\cdot\nabla\varphi=\mbox{div}(m(\varphi)\nabla\mu),\label{sy1}\\
& \mu=-K\ast\varphi+F^\prime(\varphi),\label{sy2}\\
& \mbox{div}(\uvec)=0, \label{sy4}
\end{align}
in \red{$Q$}, subject to the boundary conditions
\begin{align}
&\uvec=\mathbf{0},\qquad m(\varphi)\nabla\mu\cdot\nvec=0,\label{sy5}
\end{align}
on $\Sigma:=\partial\Omega\times (0,T)$, and to the initial conditions
\begin{align}
& \uvec(0)=\uvec_0,\qquad\varphi(0)=\varphi_0,
\label{sy6}
\end{align}
in $\Omega$. Here, $\nvec$ stands for the outward \red{unit} normal to the boundary $\partial\Omega$ of $\Omega$,
while $\uvec_0$ and $\varphi_0$ are given \blue{functions}.

\red{Problem \eqref{sy3}--\eqref{sy6} constitutes the \emph{state system} of the
control problem to be investigated below. A slightly different version thereof}  has firstly been analyzed in \cite{FGR} under rather general assumptions
on $\nu$, $m$, $J$ and $F$ (see also \cite{CFG,FG1,FG2,FGK} for more restrictive assumptions). More precisely, the mobility degenerates at the pure phases $\varphi=\pm 1$, and $F$ is a bounded (smooth) potential, defined on $(-1,1)$, whose derivatives are unbounded (i.e., a so-called singular potential). In particular, $m$ and $F$ can have the following form:
\begin{equation}
\label{physmF}
m(s)=1-s^2,\qquad F(s)= (1+s)\ln (1+s) + (1-s)\ln (1-s), \qquad s\in (-1,1).
\end{equation}
In \cite{FGR}, the existence of a global weak solution \red{was} established for a constant viscosity,
but the same argument can easily be extended
to nonconstant \red{viscosities} as well. \red{Uniqueness and existence of a strong solution
are more delicate issues, and restricted to the two-dimensional case}. The former was analyzed in \cite{FGG}, proving a conditional weak-strong uniqueness, i.e., by supposing that a strong solution exists. The existence of a strong solution has been much harder to prove. This was done in the more recent contribution \cite{FGGS} by using
a time-discretization scheme combined with a suitable approximation of $m$ and $F$.

In this paper, \red{we aim to study optimal control problems for the state system \eqref{sy3}--\eqref{sy6}, which,
in order to have a well-defined control-to-state operator, postulates the unique solvability of the state system \blue{itself}. Also,
the investigation of the differentiability properties of the control-to-state operator requires that
the solution to the state system be sufficiently regular. Both requirements make it necessary to restrict
the analysis to the spatially two-dimensional case. For this case,
we can exploit} the existence of a unique strong solution in order to formulate an optimal distributed control problem which
is similar to the one analyzed in \cite{FRS} under the more restrictive assumptions that  $m$ is constant and $F:\mathbb{R}\to \mathbb{R}$ is smooth with a
polynomially controlled growth. This is not just a minor generalization,
 since it requires a \red{considerable} technical effort, and, \blue{besides}, accounts
for choices of $m$ and $F$ which are physically more relevant.
A similar problem  was originally considered in \cite{RS} \red{in the spatially
three-dimensional case} for a convective nonlocal
Cahn--Hilliard equation with degenerate mobility and singular potential\red{, where the control was given by the velocity itself.}
However, the assumptions \red{in \cite{RS}} were more restrictive than the present ones. Indeed, the authors only considered solutions which are uniformly separated from the pure phases. Here, we do not use this property, so the initial datum can even \red{represent} a pure phase. This is possible because the system is conveniently reformulated in a more general form, following the approach devised in \cite{EG} for the Cahn--Hilliard equation with degenerate mobility, singular potential and the standard chemical potential \eqref{GL}. It is worth observing that for such an equation, and also for the corresponding system \eqref{sy01}--\eqref{sy03}, only the existence of a global weak solution has been proven \blue{so far} (cf. \cite{B}).

Let us now introduce the control problem we are interested in (see \cite{FRS}).

\vspace{5mm}\noindent
\textbf{(CP)} \quad Minimize the tracking type cost functional
\begin{align}
\mathcal{J}(y,\vvec)&:=\frac{\beta_1}{2}\Vert\uvec-\uvec_Q\Vert_{L^2(Q)^2}^2+\frac{\beta_2}{2}\Vert\varphi-\varphi_Q\Vert_{L^2(Q)}^2
+\frac{\beta_3}{2}\Vert\uvec(T)-\uvec_\Omega\Vert_{L^2(\Omega)^2}^2\nonumber\\[1mm]
&\hspace*{5mm}+\frac{\beta_4}{2}\Vert\varphi(T)-\varphi_\Omega\Vert_{L^2(\Omega)}^2+
\frac{\gamma}{2}\Vert\vvec\Vert_{L^2(Q)^2}^2\,,\label{costfunct}
\end{align}
where $y:=[\uvec,\varphi]$ solves the state system \red{\eqref{sy3}}--\eqref{sy6}.

\vspace{5mm}
\noindent
Here, the quantities $\uvec_Q\in L^2(0,T;G_{div})$, $\varphi_Q\in L^2(Q)$, $\uvec_\Omega\in G_{div}$, and $\varphi_\Omega\in L^2(\Omega)$,
are given target functions, while $\beta_i$, $i=1,\dots, 4$, and $\gamma$ are some fixed nonnegative constants that do not vanish simultaneously. \red{Moreover,}
$G_{div}$ is the classical Navier--Stokes type space (see, e.g., \cite{T}), that is,
\begin{equation*}
G_{div}:=\overline{\big\{\uvec\in C^\infty_0(\Omega)^2:\mbox{div}(\uvec)=0\big\}}^{L^2(\Omega)^d}.
\end{equation*}
\blue{The} control $\vvec$ is supposed to belong to a convenient closed, bounded and convex subset (see below) of the space of controls
$L^2(0,T;G_{div})$.

We remind that optimal control problems for the Cahn--Hilliard--Navier--Stokes system with \eqref{GL} have
recently been studied \red{for the spatially three-dimensional case, where, however, the time-discretized
version case was considered (see \cite{HHK,HHKW,HW3,HW,HW1,HW2}). We also refer to the recent
contributions \cite{CGS1,CGS2} for a treatment of the control by the velocity of convective Cahn--Hilliard systems
with dynamic boundary conditions  in three dimensions}.

The plan of this paper \red{is} as follows. In Section~\ref{singpotdm}, we introduce notation, the basic assumptions, and the notion
of weak \blue{solutions to the state system. Then
we report the existence theorem mentioned above}. \blue{Also, we  state the existence and uniqueness
result on strong solutions to the state system for the case $d=2$, which is fundamental for the
control problem,} and the related hypotheses.
Section~\ref{stability} is devoted to establish some global stability estimates that are crucial to analyze
the control problem \textbf{(CP)}. This
is studied in Section~\ref{optimal control}: first, we prove in a standard way the existence of an optimal control; then we show the Fr\'echet differentiability of the control-to-state map \red{in suitable Banach spaces (where} the stability estimates plays an essential role). Finally, we establish first-order optimality conditions.

\red{Throughout the entire paper, we will repeatedly use Young's inequality
\begin{equation}\label{Young}
a\,b\,\le\,\delta a^2+ \frac 1 {4\delta}b^2 \quad \mbox{ for all }  a,b\in\mathbb{R} \mbox{ and } \delta>0\,,
\end{equation}
and we employ the following notational convention for the use of constants
in estimates: the letter \,$C$\, denotes a generic positive constant depending only on the data of the respective
problem; the use of subscripts like in $\,C_{m,K}$\, signals that the constant depends in a bounded way on the quantities
occurring in the subscript (in this case, $m\,$ and $\,K$), in particular. In any case, the meaning will be clear and no confusion will arise. }

\section{Notation and known results \red{for the state system}}
\label{singpotdm}
\setcounter{equation}{0}

We set \,$H:=L^2(\Omega)$, $V:=H^1(\Omega)$, and denote by \,$\|\cdot\|$ and \,$
(\cdot,\cdot)$ the norm and the inner product, respectively, in both $H$
and $G_{div}$, as well as in $L^2(\Omega)^2$ and $L^2(\Omega)^{2\times 2}$.
The notation $\langle\cdot,\cdot\rangle_X$ and $\Vert\cdot\Vert_X$ will stand for
the duality pairing between a (real) Banach space $X$ and its dual $X^{\prime}$, and for the norm of $X$, respectively.
The space
\begin{equation*}
V_{div}:=\overline{\big\{\uvec\in C^\infty_0(\Omega)^2:\mbox{div}(\uvec)=0\big\}}^{H^1(\Omega)^d}
\end{equation*}
is endowed with the scalar product
\begin{equation*}
(\uvec,\vvec)_{V_{div}}=(\nabla \uvec,\nabla \vvec)=2(D\uvec,D\vvec),\qquad\forall\, \uvec,\vvec\in V_{div}.
\end{equation*}
Let us also recall the definition of the Stokes operator  $S:D(S)\cap
G_{div}\to G_{div}$ in the case of \red{the} no-slip boundary
condition \eqref{sy5}$_1$, i.e., $S=-P\Delta$ with domain $D(S)=H^2(\Omega)^d\cap V_{div}$,
where $P:L^2(\Omega)^d\to G_{div}$ is the Leray projector (see \cite{T}). Notice that we have
$$(S\uvec,\vvec)=(\uvec,\vvec)_{V_{div}}=(\nabla \uvec,\nabla \vvec),\qquad\forall\, \uvec\in D(S),\quad\forall\, \vvec\in V_{div}.$$
We also recall that $S^{-1}:G_{div}\to G_{div}$ is a self-adjoint \red{and} compact operator in $G_{div}$,
and the spectral theorem entails the existence of a sequence of eigenvalues $\lambda_j$ with $0<\lambda_1\leq\lambda_2\leq\cdots$ and $\lambda_j\to\infty$, and a corresponding family of eigenfunctions $\wvec_j\in D(S)$, which is orthonormal in $G_{div}$ and \red{satisfies} $S\wvec_j=\lambda_j\wvec_j$ \red{for all $j\in\mathbb{N}$}.
We also recall Poincar\'{e}'s inequality
\begin{equation*}
\lambda_1\,\Vert\uvec\Vert^2\leq\Vert\nabla\uvec\Vert^2,\qquad\forall\,\uvec\in V_{div},
\end{equation*}
and two other inequalities, which are valid in two dimensions of space and will be used repeatedly  in the course of our analysis\red{, namely,} the following special case of the Gagliardo--Nirenberg inequality (see, e.g., \cite{BIN}),
\begin{equation}
\label{GN}
\Vert v\Vert_{L^{2q}(\Omega)}\,\leq\,\widehat C_2\,\Vert v\Vert^{1/q}\,\Vert v\Vert_V^{1-1/q},\qquad\forall\, v\in V,
\qquad 2\leq q<\infty,
\end{equation}
and  Agmon's inequality (see \cite{AG})
\begin{equation}
\Vert v\Vert_{L^\infty(\Omega)}\,\leq\,\widehat C_3\,\Vert v\Vert^{1/2}\,
\Vert v\Vert_{H^2(\Omega)}^{1/2},\qquad\forall\, v\in H^2(\Omega).
\label{Agmon}
\end{equation}
In these inequalities, the positive constant $\widehat C_2$ depends on $q$ and on $\,\Omega \subset\mathbb{R}^2$,
 while the positive constant $\widehat C_3$ depends only on $\,\Omega$.

The trilinear form $\,b\,$ appearing in the weak formulation of the
Navier--Stokes equations is defined as usual, namely,
\begin{equation*}
b(\uvec,\vvec,\wvec):=\int_{\Omega}(\uvec\cdot\nabla)\vvec\cdot \wvec \,dx\qquad\forall\, \uvec,\vvec,\wvec\in V_{div}\,.
\end{equation*}
The associated bilinear operator $\mathcal{B}$ from $V_{div}\times
V_{div}$ into $V_{div}^{\prime}$ is defined by $\langle\mathcal{B}%
(\uvec,\vvec),\wvec\rangle:=b(\uvec,\vvec,\wvec)$, for all $\uvec,\vvec,\wvec\in V_{div}$.
We set $\mathcal{B}\,\uvec:=\mathcal{B}(\uvec,\uvec)$, for every $\uvec\in V_{div}$.
We recall the well-known identity
$$b(\uvec,\wvec,\vvec)\,=\,-\,b(\uvec,\vvec,\wvec) \qquad\forall\,\uvec,\vvec,\wvec\in
V_{div},$$
and the two-dimensional inequality
\begin{align*}
&|b(\uvec,\vvec,\wvec)|\,\leq\, \widehat C_1\,\|\uvec\|^{1/2}\,\|\nabla \uvec\|^{1/2}\,
\|\nabla \vvec\|\,\|\wvec\|^{1/2}\,\|\nabla \wvec\|^{1/2}\qquad\forall\, \uvec,\vvec,\wvec\in V_{div},
\end{align*}
with a constant $\widehat C_1>0$ that depends only on $\Omega$.

We now state the assumptions which ensure the existence of a global weak solution. \red{Although weak
solutions do not play a role for our control problem, we have decided to include the corresponding
existence result for the sake of giving the reader a complete picture
of the well-posedness results known for the state system. In particular, we include the result for the
case of three dimensions of space, noting that this result is too weak for control purposes. We
make the following assumptions:}

\begin{description}
\item[(V)] The viscosity $\nu $ is Lipschitz continuous on $[-1,1]$, and there exists some $\nu _{1}>0$ such that
\begin{equation*}
\nu _{1}\leq \nu (s)\,,\quad {\color{black}\forall s\in [-1,1]\,.}
\end{equation*}
\item[(K)]
$K(\cdot-x)\in W^{1,1}(\Omega)$ for almost every $x\in\Omega$, \red{and it holds that}\,  $K(x)=K(-x)$ and
\begin{align}
&
\sup_{x\in\Omega}\int_\Omega|K(x-y)|dy<\infty\,, \qquad
\sup_{x\in\Omega}\int_\Omega|\nabla K(x-y)|dy<\infty\,.\nonumber
\end{align}
\item[(H1)] The mobility satisfies
$m\in C^1([-1,1])$, $m\geq 0$, and
$m(s)=0$ if and only if $s=-1$ or $s=1$. Moreover, there exists some $\epsilon_0>0$
such that $m$ is nonincreasing in $[1-\epsilon_0,1]$ and nondecreasing in $%
[-1,-1+\epsilon_0]$.
\item[(H2)]  $F\in C^2(-1,1)$ \,and\, $\lambda:=mF^{\prime\,\prime}\in C\left([-1,1]\right)$.
\item[(H3)] There exists some $\epsilon_0>0$
such that $F^{\prime\prime}$ is nonincreasing in $[1-\epsilon_0,1]$ and nondecreasing in $%
[-1,-1+\epsilon_0]$.
\item[(H4)]  There exists some $c_{0}>0$ such that
\begin{equation*}
F^{\,\prime \prime }(s)\geq c_{0}\,,\qquad \forall\, s\in (-1,1)\,\,.
\end{equation*}
\item[(H5)] There exists some $\alpha_{0}>0$ such that
\begin{equation*}
\lambda(s)\geq \alpha _{0}\,,\qquad \forall\, s\in \lbrack
-1,1]\,.
\end{equation*}
\end{description}
\noindent
We also recall that if the mobility degenerates, then the notion of weak solution
must be formulated in a suitable way (cf. \cite{EG}, see also  \cite{FGR}).
\begin{defn}\label{wfdef}
Let $\uvec_0\in G_{div}$ \red{and} $\varphi_0\in \blue{L^\infty(\Omega)}$ with $F(\varphi_0)\in L^1(\Omega)$ and $%
\vvec\in L^2(0,T;V_{div}^{\prime})$ \red{be given}. A couple $%
[\uvec,\varphi]$ is \red{called} a weak solution to \eqref{sy3}--\eqref{sy6} on $[0,T]$ \red{if
and only if the following conditions hold true:}
\begin{itemize}
\item $\uvec$ and $\varphi$ satisfy
\begin{align*}
&\uvec\in L^{\infty}(0,T;G_{div})\cap L^2(0,T;V_{div}), \\
&\uvec_t\in L^{4/3}(0,T;V_{div}^{\prime})\qquad\mbox{if}\quad d=3, \\
&\uvec_t\in L^2(0,T;V_{div}^{\prime})\qquad\mbox{if}\quad d=2, \\
&\varphi\in L^{\infty}(0,T;H)\cap L^2(0,T;V)  \red{ \cap L^\infty(Q) }, \\
&\varphi_t\in L^2(0,T;V^{\prime}),\\
&|\varphi(x,t)|\leq 1\quad\mbox{ for a.e. }%
(x,t)\in Q;
\end{align*}
\item for every $\wvec\in V_{div}$, every $\psi\in V$, and \red{almost every} $%
t\in(0,T)$, we have
\begin{align*}
&\red{\langle \uvec_t,\wvec\rangle_{V_{div}}}+2\,\left(\nu\left(\varphi\right)D\uvec,D\wvec\right)+b(\uvec,\uvec,\wvec)\,=\,
-\big((K\ast\varphi)\nabla\varphi,\wvec\big)+\red{\langle \vvec,\wvec\rangle_{V_{div}}},\\[1mm]
&\red{\langle\varphi_t,\psi\rangle_V}+\int_\Omega
m(\varphi)F^{\prime\prime}(\varphi)\nabla\varphi\cdot\nabla\psi\,dx
-\int_\Omega m(\varphi)(\nabla
K\ast\varphi)\cdot\nabla\psi\,dx\,=\,(\uvec\varphi,\nabla\psi);
\end{align*}
\item the initial conditions $\uvec(0)=\uvec_0$ and $\varphi(0)=\varphi_0$ are fulfilled.
\end{itemize}
\end{defn}

\noindent \red{We observe that the} regularity properties of the weak solution imply the weak continuity
$\uvec\in C_w([0,T];G_{div})$ and $\varphi\in C_w([0,T];H)$. Therefore, the
initial conditions are meaningful.

We now report the result shown in \cite{FGR}. In this connection, we point out that there the viscosity $\nu$ was
assumed to be constant just to avoid technicalities; however, the assertion of the theorem
still holds true if $\nu$ satisfies only \textbf{(V)} (see also \cite{FGG} for further details).

\begin{thm}
\label{uniqthmdeg}
Assume that  \textbf{(V)}, \textbf{(K)} and \textbf{(H1)}--\textbf{(H5)} are satisfied. Let
$\uvec_0\in G_{div}$ and $\varphi_0\in L^\infty(\Omega)$ with $F(\varphi_0)\in L^1(\Omega)$ and $M(\varphi_0)\in
L^1(\Omega)$ \red{be given}, where \,$M\in C^2(-1,1)$\, solves \,$m(s)M^{\prime\prime}(s)=1$
\,for all\, $s\in(-1,1)$\, with \,$M(0)=M^{\prime}(0)=0$. \red{Assume also that}
$\vvec\in L^2_{loc}([0,\infty);V_{div}^{\prime})$.
Then, for every $T>0$, \red{the state} system \eqref{sy3}--\eqref{sy6} admits a weak solution $[\uvec,\varphi]$
on $[0,T]$ such that \red{the mean values satisfy}\, $\overline{\varphi}(t)=\overline{\varphi}_0$ for all $t\in[0,T]$.
\blue{If
$d=2$, then} the weak solution $[\uvec,\varphi]$ satisfies the energy equation
\begin{align}
&
\frac{1}{2}\frac{d}{dt}
\big(\Vert \uvec\Vert^2+\Vert\varphi\Vert^2\big)+\int_\Omega m(\varphi)F''(\varphi)|\nabla\varphi|^2\,\red{dx}
\,+\,2\Vert \sqrt{\nu(\varphi)} D\uvec\Vert^2\nonumber\\
&=\int_\Omega m(\varphi)\red{(\nabla K\ast\varphi)} \cdot\nabla\varphi\red{\,dx}\,
 - \int_\Omega \red{(K\ast\varphi)}\,
\uvec\cdot\nabla\varphi\,\red{dx\,}+\red{\langle \vvec,\uvec\rangle_{V_{div}}\,\quad\mbox{for a.e. }\,t\in (0,T).}
\label{energeq}
\end{align}
\blue{If $d=3$}, then $[\uvec,\varphi]$ satisfies the energy inequality
\begin{align}
&\frac{1}{2} \left(\Vert \uvec(t)\Vert^2+\Vert\varphi(t)\Vert^2\right)
+ 2 \int_0^t\Vert \sqrt{\nu(\varphi)} D\uvec\Vert^2\red{(s)\,ds}\,
 + \int_0^t\int_\Omega m(\varphi) F''(\varphi)|\nabla\varphi|^2\red{\,dx\,ds}\nonumber\\
&\leq\frac{1}{2}\big(\Vert \uvec_0\Vert^2+\Vert\varphi_0\Vert^2\big)
+\int_0^t\int_\Omega m(\varphi)\red{( \nabla K\ast\varphi)} \cdot\nabla\varphi\,\red{dx\,ds}
\nonumber\\
&\quad - \int_0^t\int_\Omega \red{(K\ast\varphi)} \uvec\cdot\nabla\varphi\,\red{dx\,ds}\,+
\int_0^t\red{\langle \vvec(s),\uvec(s)\rangle_{V_{div}}(s)\,ds}\,,
\qquad\forall \,\red{t\in (0,T]}.
\label{energineq}
\end{align}
\end{thm}

\red{As noted above, the notion of weak solution does not suffice for purposes of optimal control theory.}
In order to introduce the notion of strong solution, we need \blue{the} slightly stronger assumption:
\begin{description}
\item[(H2*)] $F\in C^{3}(-1,1)$, and \,$\lambda :=mF^{\prime\prime }\in C^{1}([-1,1])$.
\end{description}
We then set (see \cite{FGGS})
\begin{equation}
B(s):=\int_{0}^{s}\lambda (\sigma )d\sigma \,,\qquad \forall s\in \lbrack -1,1]\,.
\label{defB}
\end{equation}

\noindent Moreover, we recall the definition \red{of the notion of admissible kernels} (see \cite[Definition 1]{BRB}):

\begin{defn}
A kernel $K\in W_{loc}^{1,1}(\mathbb{R}^{d})$ is called admissible if and only if the following conditions are satisfied:
\begin{description}
\item[(K1)] $K\in C^{3}(\mathbb{R}^{d}\backslash \{0\})$;
\item[(K2)] $K$ is radially symmetric, $K(x)=\tilde{K}(|x|)$, and \,$\tilde{K}$
is nonincreasing;
\item[(K3)] $\tilde{K}^{\prime \prime }(r)$ and $\tilde{K}^{\prime }(r)/r$
are monotone on $(0,r_{0})$\, for some \,$r_{0}>0$;
\item[(K4)] $|D^{3}K(x)|\leq C_{d}|x|^{-d-1}$ \,for some\, $C_{\ast }>0$.
\end{description}
\end{defn}

For the readers' convenience, we report the following lemma.

\begin{lem}
\label{admiss} {\rm (cf. \cite[Lemma 2]{BRB})}\,\, Let $K$ be admissible. Then, for
every $p\in (1,\infty )$, there exists some $C_{p}>0$ such that
\begin{equation}
\label{yeah}
\red{\Vert \nabla(\nabla K\ast \psi)\Vert} _{L^{p}(\Omega )^{d\times d}}\leq C_{p}\Vert \psi \Vert
_{L^{p}(\Omega )}\,\qquad \forall\,\psi \in L^{p}(\Omega )\,,
\end{equation}
\red{where}\,  $C_{p}=C_{\ast }p$ for $p\in \lbrack
2,\infty )$ and $C_{p}=C_{\ast }p/\left( p-1\right) $ for $p\in \left(
1,2\right) $, with some constant $C_{\ast }>0$ which is independent of $p.$
\end{lem}

\red{After these preliminaries, \blue{we assume} for the remainder of this paper that $d=2$, which implies that
\eqref{GN} and \eqref{Agmon} are valid and the embedding $V\subset L^p(\Omega)$ is continuous and compact for
$1\le p<+\infty$.}  \blue{We now report the notion of strong solution introduced in \cite{FGGS}}:

\begin{defn}
\label{strongsol} Assume that \,$\boldsymbol{u}_{0}\in V_{div}$, $\varphi _{0}\in
V\cap C^{\beta }(\overline{\Omega })$\, for some \,$\beta \in (0,1)$, and\,
$\boldsymbol{v}\in L^{2}(0,T;G_{div})$\, \red{are given}.
A weak solution $[\boldsymbol{u},\varphi ]$ to the state system \eqref{sy3}--\eqref{sy6}
on $[0,T]$ corresponding to $[\boldsymbol{u}_{0},\varphi _{0}]$
is called a strong solution if and only if it holds that
\begin{align}
& \boldsymbol{u}\in L^{\infty }\left( 0,T;V_{div}\right) \cap
L^{2}(0,T;H^{2}\left( \Omega \right) ^{2})\,,\qquad \boldsymbol{u}_{t}\in
L^{2}\left( 0,T;G_{div}\right) \,,  \label{stt1} \\
& \varphi \in L^{\infty }(0,T;V)\cap L^{2}(0,T;H^{2}(\Omega ))\,,\qquad
\varphi _{t}\in L^{2}(0,T;H)\,,  \label{stt1bis} \\
& \boldsymbol{u}_{t}-2\mbox{\rm div}\left( \nu (\varphi )D\boldsymbol{u}\right)
+(\boldsymbol{u}\cdot \nabla )\boldsymbol{u}+\nabla \pi =-(K*\varphi)\nabla \varphi +%
\boldsymbol{v}\,  \red{\quad\mbox{a.e. in $\,Q$,}}\label{stt2}  \\
& \varphi _{t}+\boldsymbol{u}\cdot \nabla \varphi =\Delta B(\varphi )-%
\mbox{\rm div}\,\big(m(\varphi )(\nabla K\ast \varphi )\big)\, \red{\quad\mbox{a.e. in $\,Q$,}} \label{stt3}  \\
& \mbox{\rm div}\,(\boldsymbol{u})=0\,\red{\quad\mbox{a.e. in $\,Q$,}}   \label{stt4}\\
&\boldsymbol{u}=\mathbf{0}\,,\quad \big[\nabla B(\varphi )-m(\varphi )(\nabla
K\ast \varphi )\big]\cdot \boldsymbol{n}=0\,\red{\quad\mbox{a.e. on $\,\Sigma$\,,}} \label{stt5}
\end{align}
\blue{for some $\pi \in L^2(0,T;V)$.}
\end{defn}

\red{We have the following result} (see \cite[Thm.3.6]{FGGS} and \cite[Rem.4.5]{FGGS}, cf. also \cite[Rem.3.7]{FGGS}).

\begin{thm}
\label{reg-thm} Let the assumptions \textbf{(V)}, \textbf{(K)}, \textbf{(H1)},
\textbf{(H2*)}--\textbf{(H5)} hold true, and assume that \red{either} $K\in
W_{loc}^{2,1}(\mathbb{R}^{2})$\, or  \,$K$ is admissible. Let
$\boldsymbol{u}_{0}\in G_{div}$\, and \,$\varphi _{0}\in V\cap L^{\infty }(\Omega )$\,
with \,$F(\varphi _{0})\in L^{1}(\Omega )$\, and \,$M(\varphi _{0})\in
L^{1}(\Omega )$\, \red{be given}, where \,$M$\, is defined as in Theorem \ref{uniqthmdeg}.
\red{Moreover, suppose that}\, $\boldsymbol{v}\in L^{2}(0,T;G_{div})$. Then, for every $T>0$,
\red{the state system} %
\eqref{sy3}--\eqref{sy6} admits a weak solution $[\boldsymbol{u},\varphi ]$
on $[0,T]$.
\red{If, in addition,} $\boldsymbol{u}_{0}\in V_{div}$ and $\varphi
_{0}\in V\cap C^{\beta }(\overline{\Omega })$ for some $\beta \in (0,1)$\red{,
then the state system} \eqref{sy3}--\eqref{sy6} admits a unique strong solution in the sense
of Definition \ref{strongsol}.
Finally, if $\varphi _{0}\in H^{2}(\Omega )$ fulfills the compatibility condition
\begin{equation}
\left[\nabla B(\varphi _{0})-m(\varphi _{0})(\nabla K\ast
\varphi _{0})\right]\cdot \boldsymbol{n}=0\,,\quad \mbox{ a.e. on }\partial \Omega \,,
\label{comp}
\end{equation}%
then the strong solution also satisfies
\begin{equation}
\varphi \in L^{\infty }(0,T;H^{2}(\Omega ))\,,\qquad \varphi _{t}\in
L^{\infty }(0,T;H)\cap L^{2}(0,T;V)\,.  \label{reg3}
\end{equation}
\red{Moreover,} there exists a continuous and nondecreasing function
$ \mathbb{Q}_1:[0,\infty )\rightarrow \lbrack 0,+\infty )$,
which only depends on $F$, $m$, $K$, $\nu $, $\Omega $, $T$, $\boldsymbol{u}_{0}$ and $\varphi _{0}$, such that
\begin{align}
& \Vert \boldsymbol{u}\Vert _{L^{\infty }\left( \left[ 0,T\right]
;V_{div}\right) \cap L^{2}(0,T;H^{2}(\Omega )^{2})}+\Vert \boldsymbol{u}%
_{t}\Vert _{L^{2}\left( \left[ 0,T\right] ;G_{div}\right) }+\Vert \varphi
\Vert _{L^{\infty }\left( \left[ 0,T\right] ;H^{2}(\Omega )\right) }
\notag \\
& +\Vert \varphi _{t}\Vert _{L^{\infty }\left( \left[ 0,T\right] ;H\right)
\cap L^{2}(0,T;V)}  \,\leq\,  \mathbb{Q}_1\left( \Vert \boldsymbol{v}\Vert
_{L^{2}(0,T;G_{div})}\right) .  \label{bound1}
\end{align}%
\end{thm}

\section{Stability of the control-to-state \red{mapping}}
\label{stability}
\setcounter{equation}{0}

We shall henceforth \red{assume that the initial data $\uvec_0$, $\varphi_0$ satisfy} the following assumptions:
\begin{description}
  \item[(H6)] $\uvec_0\in V_{div},\;\; \varphi_0\in H^2(\Omega) \red{\mbox{ satisfies }} \eqref{comp},
  \;\;  F(\varphi_0)\in L^1(\Omega),\;\;  M(\varphi_0)\in L^1(\Omega)$.
\end{description}
Then we set
\begin{align}
&\mathcal{V}:=L^2\left(0,T;G_{div}\right),\nonumber\\
&\mathcal{H}:=\,\big[H^1(0,T;G_{div})\cap C^0([0,T];V_{div})\cap L^2(0,T;H^2(\Omega)^2)\big]\nonumber
\\
&\hspace*{12mm}\times \big[C^1([0,T];H)
\cap H^1(0,T;V)\cap L^\infty(0,T;H^2(\Omega))\big].\label{controlmap1}
\end{align}
On account of Theorem \ref{reg-thm}, the control-to-state mapping
 \begin{align}
&\mathcal{S}:\mathcal{V}\to\mathcal{H},\qquad\vvec\in\mathcal{V}\mapsto\mathcal{S}(\vvec):=[\uvec,\varphi]\in\mathcal{H},\label{controlmap2}
\end{align}
where $[\uvec,\varphi]$ is the (unique) strong solution to \eqref{sy3}--\eqref{sy6} corresponding to
the fixed initial data $\uvec_0$, $\varphi_0$ and to the control $\vvec\in\mathcal{V}$, is well defined and locally bounded.
We now establish some global stability estimates for the strong solutions to the state system
\eqref{sy3}--\red{\eqref{sy6}. In doing this, we} can argue formally, since the arguments can be made rigorous within the approximation scheme devised in \cite{FGGS}.
The first result is the following.
\begin{lem}\label{stab-est-res1}
Let the assumptions \textbf{(V)}, \textbf{(K)}, \textbf{(H1)}, and
\textbf{(H2*)}--\textbf{(H6)} hold true, and suppose that $K\in W^{2,1}_{loc}(\mathbb{R}^2)$ or that $K$ is admissible.
Assume moreover that controls $\vvec_i\in\mathcal{V}$,
   $i=1,2$, are given and that $[\uvec_i,\varphi_i]:=\mathcal{S}(\vvec_i)$, $i=1,2$, are the associated solutions to
the state system   \eqref{sy3}--\red{\eqref{sy6}}. Then there exists a continuous function\, $\mathbb{Q}_2:[0,\infty)^2\to[0,\infty)$, which
is nondecreasing in both its arguments and depends only on the data $F$, $m$, $K$, $\nu_1$, $\Omega$,
$T$, $\uvec_0$ and $\varphi_0$, such that we have the estimate
\begin{align}
&
\Vert\uvec_2-\uvec_1\Vert_{C^0([0,t];G_{div})}^2
\,+\,\Vert\uvec_2-\uvec_1\Vert_{L^2(0,t;V_{div})}^2\,+\,\Vert\varphi_2-\varphi_1\Vert_{C^0([0,t];H)}^2
\,+\,\Vert\varphi_2-\varphi_1\Vert_{L^2(0,t;V)}^2\nonumber
\\[1mm]
\label{stabi1}
&
\leq\,\mathbb{Q}_2\big(\Vert\vvec_1\Vert_{L^2(0,T;G_{div})},\Vert\vvec_2\Vert_{L^2(0,T;G_{div})} \big)\,\Vert\vvec_2-\vvec_1\Vert_{L^2(0,T;V_{div}^\prime)}^2\,  \blue{\quad \forall\,t\in (0,T]}.
\end{align}

\end{lem}

\begin{proof}
In this proof, we omit the explicit dependence on time for the sake of simplicity.
Let us test the difference between \eqref{stt3}, written for each of the two solutions, by $\varphi:=\varphi_2-\varphi_1$ in $H$. Taking
\eqref{stt4} into account, we obtain the  differential identity
\begin{align}
&\frac{1}{2}\frac{d}{dt}\Vert\varphi\Vert^2+\left(\uvec\cdot\nabla\varphi_2,\varphi\right)
+\left(\nabla\left(B\left(\varphi_2\right)-B\left(\varphi_1\right)\right),\nabla\varphi\right)\nonumber\\
&=
\left(\left(m(\varphi_2)-m(\varphi_1)\right)\left(\nabla K\ast\varphi_2\right)+m(\varphi_1)\nabla K\ast\varphi,\nabla\varphi\right)\,,
\label{diffid4}
\end{align}
where $\uvec = \uvec_1 - \uvec_2$. \red{Using {\bf (H5)}, the mean value theorem, the Gagliardo--Nirenberg inequality
\eqref{GN}, the boundedness of $\varphi_2$, the regularity result \eqref{reg3}, and Young's inequality, we find that the}
third term on the left-hand side of \eqref{diffid4} can be estimated \red{as follows} (cf. \eqref{defB}):
\begin{align}
&\left(\nabla\left(B\left(\varphi_2\right)-B\left(\varphi_1\right)\right),\nabla\varphi\right)
\,=\,\left(\lambda(\varphi_2)\nabla\varphi+\left(\lambda(\varphi_2)-\lambda(\varphi_1)\right)\nabla\varphi_1,\nabla\varphi\right)
\nonumber\\
&\geq\,\red{\alpha_0}\,\Vert\nabla\varphi\Vert^2-k_1\,\Vert\varphi\Vert_{L^4(\Omega)}\,\Vert\nabla\varphi_1\Vert_{L^4(\Omega)^2}\,
\Vert\nabla\varphi\Vert
\nonumber\\
&\geq \red{\alpha_0}\,\Vert\nabla\varphi\Vert^2-C \,\Vert\varphi_1\Vert_{H^2(\Omega)}\, \left(\Vert\varphi\Vert +
\Vert\varphi\Vert^{1/2}\Vert\nabla\varphi\Vert^{1/2}\right)\Vert\nabla\varphi\Vert\nonumber\\
&\geq\,\red{\frac{\alpha_0}{2}}\,\Vert\nabla\varphi\Vert^2- \mathbb{Q} \,\Vert\varphi\Vert^2,\label{est16}
\end{align}
where $k_1:=\Vert \lambda'\Vert_{C([-1,1])}$. \red{Here, and in the remainder of this proof, $\mathbb{Q}$ stands for a function having similar properties as
the function $\mathbb{Q}_2$ in the statement of  the theorem}.

Concerning the right-hand side of \eqref{diffid4}, we have\red{, setting
\begin{equation}\label{mbounds}
m_\infty:=\max_{\varphi\in [-1,1]} \,|m(\varphi)| \quad\mbox{and}\quad
m_\infty':=\max_{\varphi\in [-1,1]}\,|m'(\varphi)|,
\end{equation}
and using the mean value theorem and Young's inequality,}
\begin{align}
&\left|\left(\left(m(\varphi_2)-m(\varphi_1)\right)\left(\nabla K\ast\varphi_2\right)
+m(\varphi_1)\nabla K\ast\varphi,\nabla\varphi\right)
\right|\notag\\
&\leq \left(m_\infty '+m_\infty\right)\,\Vert\nabla K\Vert_{\red{L^1(\Omega)}}\,\Vert\varphi\Vert\,
\Vert\nabla\varphi\Vert\nonumber\\
&\leq\frac{\red{\alpha_0}}{4}
\,\Vert\nabla\varphi\Vert^2+C\,\Vert\varphi\Vert^2.
\end{align}
\red{Moreover, invoking \eqref{reg3}, as well as H\"older's and Young's inequalities, we readily find that}
\begin{equation}
\left|\left(\uvec\cdot\nabla\varphi_2,\varphi\right)\right|
\,\leq\,\Vert\uvec\Vert_{L^4(\Omega)^2}\,\Vert\nabla\varphi_2\Vert_{L^4(\Omega)^2}\,\Vert\varphi\Vert
\,\leq\, \frac{\nu_1}{8}\,\Vert\nabla\uvec\Vert^2
+ \mathbb{Q}\,\Vert\varphi\Vert^2.\label{est17}
\end{equation}
Hence, \red{combining \eqref{diffid4}--\eqref{est17}}, we obtain that
\begin{align}
&\frac{1}{2}\,\frac{d}{dt}\,\Vert\varphi\Vert^2
+\frac{\red{\alpha_0}}{4}
\,\Vert\nabla\varphi\Vert^2\,\leq \,\mathbb{Q}\,\Vert\varphi\Vert^2+\frac{\nu_1}{8}\,\Vert\nabla\uvec\Vert^2
\red{\quad\mbox{a.e. in }\,(0,T)}.\label{est18}
\end{align}
On the other hand, by testing the difference of \eqref{stt2}, written for each of the two solutions, by $\uvec$
in $G_{div}$, and arguing as in the proof of \cite[Thm. 7]{FGG}, the following differential inequality can be deduced:
\begin{align}
& \frac{1}{2}\,\frac{d}{dt}\,\Vert \uvec\Vert ^{2}+\frac{\nu _{1}}{4}\,\Vert \nabla
\uvec \Vert ^{2}\,\leq\, \frac{\red{\alpha_0}}{\red{8}}\,\Vert \nabla \varphi \Vert
^{2}  \notag \\
& +C\left(1+\Vert \nabla \uvec_{2}\Vert ^{2}\Vert \uvec_{2}\Vert _{H^{2}\red{(\Omega)}}^{2}+\Vert
\varphi _{1}\Vert _{L^{4}\red{(\Omega)}}^{2}+\Vert \varphi _{2}\Vert _{L^{4}\red{(\Omega)}}^{2}\right)\Vert
\varphi \Vert ^{2} \notag\\
&+ C\,\Vert \nabla \uvec_{1}\Vert ^{2}\Vert \uvec\Vert ^{2} +\frac{1}{\nu_1}\Vert\vvec\Vert_{V_{div}^\prime}^2\,
\red{\quad\mbox{a.e. in \,$(0,T)$}}.
\label{diffineq1}
\end{align}%
Therefore, we get
\begin{equation}
\frac{1}{2}\frac{d}{dt}\Vert\uvec\Vert^2+\frac{\nu_1}{4}\Vert\nabla\uvec\Vert^2\leq
\red{\frac{\alpha_0}{8}}\,\Vert \nabla \varphi \Vert^{2}
+\Lambda_1\left(\Vert\varphi\Vert^2+ \Vert\uvec\Vert^2\right)+\frac{1}{\nu_1}\Vert\vvec\Vert_{V_{div}^\prime}^2
\quad\,\red{\mbox{a.e. in \,$(0,T)$}},
\label{est19}
\end{equation}
where $\vvec:=\vvec_2-\vvec_1$ and
$$
\Lambda_1 := C\left(1+ \mathbb{Q} + \Vert \nabla \uvec_{2}\Vert ^{2}\Vert \uvec_{2}\Vert _{H^{2}\red{(\Omega)}}^{2}+\Vert
\varphi _{1}\Vert _{L^{4}\red{(\Omega)}}^{2}+\Vert \varphi _{2}\Vert _{L^{4}\red{(\Omega)}}^{2}\right)\in L^1(0,T)\,.
$$
By adding \eqref{est18} to \eqref{est19}, and applying Gronwall's lemma to
the resulting differential inequality, we finally obtain the asserted stability estimate \eqref{stabi1}.
\end{proof}

The following higher-order stability estimate for the solution component $\varphi$
will be crucial for the proof of the Fr\'echet differentiability of the control-to-state mapping.
In order to achieve this, we need to strengthen \red{the hypotheses} \textbf{(H1)} and \textbf{(H2*)}
\red{somewhat. More precisely, we postulate the following conditions:}
\begin{description}
  \item[\textbf{(H1*)}] The mobility satisfies \textbf{(H1)} and also $m\in C^2\left([-1,1\right])$.
  \item[\textbf{(H2**)}] $F\in C^4(-1,1)$ and $\lambda:=mF''\in C^2\left([-1,1]\right)$.
\end{description}

Moreover, we need the following lemma to handle some boundary terms.

\begin{lem}\label{trace-product}
Let $\phi,\psi\in H^{1/2}(\partial\Omega)\cap L^\infty(\partial\Omega)$. Then $\phi\psi\in H^{1/2}(\partial\Omega)\cap L^\infty(\partial\Omega)$, and we have
\begin{align*}
&\Vert\phi\psi\Vert_{H^{1/2}(\partial\Omega)}
\le \Vert\phi\Vert_{L^\infty(\partial\Omega)}\Vert\psi\Vert_{H^{1/2}(\partial\Omega)}
+\Vert\psi\Vert_{L^\infty(\partial\Omega)}\Vert\phi\Vert_{H^{1/2}(\partial\Omega)}.
\end{align*}
\end{lem}
\begin{proof}
The proof is an immediate consequence of the definition of the space $H^{1/2}(\partial\Omega)$
with seminorm given by
\begin{align}
&|\phi|_{H^{1/2}(\partial\Omega)}^2=\int_{\partial\Omega}\int_{\partial\Omega}
\frac{|\phi(x)-\phi(y)|^2}{|x-y|^2}d\Gamma(x) d\Gamma(y),\label{H^1/2-seminorm}
\end{align}
where $d\Gamma(\cdot)$ is the surface measure on $\partial\Omega$ (see, e.g., \cite[Chapter IX, Section 18]{DiB}).
\end{proof}

We have \red{the following stability result.}

\begin{lem}\label{stablem2}
Let the assumptions \textbf{(V)}, \textbf{(K)}, \textbf{(H1*)},
\textbf{(H2**)}, \textbf{(H3)}-\textbf{(H6)} hold true, and suppose that $K\in W^{2,1}_{loc}(\mathbb{R}^2)$ or that $K$ is admissible.
Then there exists a continuous function $\mathbb{Q}_3:[0,\infty)^2\to[0,\infty)$, which
is nondecreasing in both its arguments and depends only on the data $F$, $m$, $K$, $\nu_1$, $\Omega$,
$T$, $\uvec_0$ and $\varphi_0$, such that we have for every $t\in (0,T]$ the estimate
\begin{align}
&\Vert\uvec_2-\uvec_1\Vert_{L^\infty\left(0,t;G_{div}\right)}^2+\Vert\uvec_2-\uvec_1\Vert_{L^2\left(0,t;V_{div}\right)}^2
+\Vert\varphi_2-\varphi_1\Vert_{L^\infty\left(0,t;V\right)}^2+
\Vert\varphi_2-\varphi_1\Vert_{L^2\left(0,t;H^2(\Omega)\right)}^2\nonumber\\
&+\Vert\varphi_2-\varphi_1\Vert_{H^1\left(0,t;H\right)}^2
\,\leq\,\mathbb{Q}_3\big(\Vert\vvec_1\Vert_{L^2(0,T;G_{div})},\Vert\vvec_2\Vert_{L^2(0,T;G_{div})} \big)\,
\Vert\vvec_2-\vvec_1\Vert_{L^2(0,T;V_{div}^\prime)}^2\,.
\label{stabest2}
\end{align}

\end{lem}

\begin{proof}
In the following, the explicit dependence on time is omitted for simplicity.
Let us take the difference between \eqref{stt3} written for each of the two solutions, and test
the resulting equation
by $\left(B\left(\varphi_2\right)-B\left(\varphi_1\right)\right)_t$ in $H$. As in the proof of the previous lemma,
we set $\uvec:=\uvec_2-\uvec_1$ and $\varphi:=\varphi_2-\varphi_1$. On account of
\eqref{defB}, we obtain \red{almost everywhere in $(0,T)$} the identity
\begin{align}\label{diffid6}
 &\frac{1}{2}\frac{d\Psi}{dt}+\big(\lambda\left(\varphi_1\right)\varphi_t,\varphi_t\big)
 =-\big(\left(\lambda(\varphi_2)-\lambda(\varphi_1)\right)\varphi_{2,t},\varphi_t\big)\nonumber\\
 &-\big(\uvec\cdot\nabla\varphi_2,\left(\lambda(\varphi_2)-\lambda(\varphi_1)\right)\varphi_{2,t}\big)
 -\big(\uvec\cdot\nabla\varphi_2,\lambda(\varphi_1)\varphi_t\big)\nonumber\\
 &-\big(\uvec_1\cdot\nabla\varphi,\left(\lambda(\varphi_2)-\lambda(\varphi_1)\right)\varphi_{2,t}\big)
 -\big(\uvec_1\cdot\nabla\varphi,\lambda(\varphi_1)\varphi_t\big)\nonumber\\
 &-\big(\left(m'(\varphi_2)-m'(\varphi_1)\right)\varphi_{2,t}\left(\nabla K\ast\varphi_2\right),\nabla\left(B(\varphi_2)-B(\varphi_1)\right)\big)\nonumber\\
 &-\big(m'(\varphi_1)\varphi_{t}\left(\nabla K\ast\varphi_2\right),\nabla\left(B(\varphi_2)-B(\varphi_1)\right)\big)\nonumber\\
 &-\big(\left(m(\varphi_2)-m(\varphi_1)\right)\left(\nabla K\ast\varphi_{2,t}\right),\nabla\left(B(\varphi_2)-B(\varphi_1)\right)\big)\nonumber\\
 &-\big(m'(\varphi_1)\varphi_{1,t}\left(\nabla K\ast\varphi\right),\nabla\left(B(\varphi_2)-B(\varphi_1)\right)\big)\nonumber\\
 &-\big(m(\varphi_1)\left(\nabla K\ast\varphi_t\right),\nabla\left(B(\varphi_2)-B(\varphi_1)\right)\big)\,=\,\sum_{j=1}^
 {\red{10}} I^{(1)}_j,
  \end{align}
where \red{the quantities $I_j^{(1)}$, $1\le j\le 10$, have obvious meaning and} the functional $\Psi$ is defined by
\begin{align}\label{Psi}
 \Psi&:=\Vert\nabla\left(B(\varphi_2)-B(\varphi_1)\right)\Vert^2
 \,-\,2\big(\left(m(\varphi_2)-m(\varphi_1)\right)\left(\nabla K\ast\varphi_2\right),\nabla\left(B(\varphi_2)-B(\varphi_1)\right)\big)
 \nonumber\\
 &\hspace*{7mm} -2\big(m(\varphi_1)\left(\nabla K\ast\varphi\right),\nabla\left(B(\varphi_2)-B(\varphi_1)\right)\big).
\end{align}
We now estimate individually all of the terms on the right-hand side of \eqref{diffid6}. \red{To this end, we note
that the mean value theorem yields that
$$|\lambda(\varphi_2)-\lambda(\varphi_1)|\,+\,\max_{0\le k\le 1}\,|m^{(k)}(\varphi_2)
-m^{(k)}(\varphi_1)|\,\le\,C_0\,|\varphi| \quad\,\mbox{a.e. in \,$Q$},
$$
with some global constant $C_0$. Moreover, we recall the continuity of the embedding $V\subset L^4(\Omega)$,
the Gagliardo--Nirenberg inequality \eqref{GN}, and the regularity properties stated in Theorem 2. Using
H\"older's and Young's inequalities, we obtain, for every $\epsilon>0$ and $\epsilon'>0$ (which will be
specified later), the following chain of estimates:}
\begin{align}
I^{(1)}_1
&\,\leq\, C_0\,\Vert\varphi\Vert_{L^4(\Omega)}\,
\Vert\varphi_{2,t}\Vert_{L^4(\Omega)}\,\Vert\varphi_t\Vert
\,\leq\,\epsilon\,\Vert\varphi_t\Vert^2+C_\epsilon\, \Vert\varphi_{2,t}\Vert_V^2\,\Vert\varphi\Vert_V^2\,,\label{est23}
\\[1mm]
I^{(1)}_2
&\,\leq\, C_0\,\Vert\uvec\Vert_{L^4(\Omega)^2}\,\Vert\nabla\varphi_2\Vert_{L^4(\Omega)^2}\,
\Vert\varphi\Vert_{L^4(\Omega)}\,\Vert\varphi_{2,t}\Vert_{L^4(\Omega)}\nonumber\\
&\,\leq\,\epsilon'\,\Vert\nabla\uvec\Vert^2+C_{\epsilon'}\,\Vert\varphi_{2,t}\Vert_V^2\,\Vert\varphi\Vert_V^2\,,\\[1mm]
I^{(1)}_3
&\,\leq\,\red{C}\,\Vert\uvec\Vert_{L^4(\Omega)^2}\,\Vert\nabla\varphi_2\Vert_{L^4(\Omega)^2}\,\Vert\varphi_t\Vert
\,\leq\, \mathbb{Q}\,\Vert\uvec\Vert^{1/2}\,\Vert\nabla\uvec\Vert^{1/2}\,\Vert\varphi_t\Vert\nonumber\\
&\,\leq\, \epsilon\,\Vert\varphi_t\Vert^2+\epsilon'\,\Vert\nabla\uvec\Vert^2+\mathbb{Q}\,\Vert\uvec\Vert^2\,,\\[1mm]
I^{(1)}_4
&\,\leq\, C_0\,\Vert\uvec_1\Vert_{L^\infty(\Omega)^2}\,\Vert\nabla\varphi\Vert\,\Vert\varphi\Vert_{L^4(\Omega)}\,
\Vert\varphi_{2,t}\Vert_{L^4(\Omega)}
\,\leq\, C\,\Vert\uvec_1\Vert_{H^2(\Omega)^2}\,\Vert\varphi_{2,t}\Vert_{V}\,\Vert\varphi\Vert_V^2\,,\\[1mm]
I^{(1)}_5
&\,\leq\,\red{C}\,\Vert\uvec_1\Vert_{L^\infty(\Omega)^2}\,\Vert\nabla\varphi\Vert\,\Vert\varphi_t\Vert
\,\leq\, \epsilon\,\Vert\varphi_t\Vert^2+C_\epsilon\,\Vert\uvec_1\Vert_{H^2(\Omega)^2}^2\,\Vert\varphi\Vert_V^2\,,\\[1mm]
I^{(1)}_6
&\,\leq\,\Vert m'(\varphi_2)-m'(\varphi_1)\Vert_{L^4(\Omega)}\,\Vert\varphi_{2,t}\Vert_{L^4(\Omega)}\,
\Vert\nabla K\ast\varphi_2\Vert_{L^\infty(\Omega)^2}\,\Vert\nabla\left(B(\varphi_2)-B(\varphi_1)\right)\Vert
\nonumber\\
&\,\leq\, C\,\Vert\varphi_{2,t}\Vert_{V}\,\Vert\varphi\Vert_V\,
\Vert\nabla\left(B(\varphi_2)-B(\varphi_1)\right)\Vert\,,\\[1mm]
I^{(1)}_7+I^{(1)}_{10}
&\,\leq\, C\,\Vert\varphi_t\Vert\,\Vert\nabla\left(B(\varphi_2)-B(\varphi_1)\right)\Vert\nonumber\\
&\,\leq\,\epsilon\,\Vert\varphi_t\Vert^2+C_\epsilon\,\Vert\nabla\left(B(\varphi_2)-B(\varphi_1)\right)\Vert^2\,,\\[1mm]
I^{(1)}_{8}
&\,\leq\, \Vert m(\varphi_2)-m(\varphi_1)\Vert_{L^4(\Omega)}\,\Vert\nabla K\ast\varphi_{2,t}\Vert_{L^4(\Omega)^2}\,
\Vert\nabla\left(B(\varphi_2)-B(\varphi_1)\right)\Vert\nonumber\\
&\,\leq\, C\,\Vert\varphi_{2,t}\Vert_{V}\,\Vert\varphi\Vert_{V}\,\Vert\nabla\left(B(\varphi_2)-B(\varphi_1)\right)\Vert\,,
\\[1mm]
I^{(1)}_{9}
&\,\leq\, C\, \Vert\varphi_{1,t}\Vert_{L^4(\Omega)}\,\Vert\nabla K\ast\varphi\Vert_{L^4(\Omega)^2}\,\Vert\nabla\left(B( \varphi_2)-B(\varphi_1)\right)\Vert\nonumber\\
&\,\leq\, C\,\Vert\varphi_{1,t}\Vert_{V}\,\Vert\varphi\Vert_{V}\,\Vert\nabla\left(B(\varphi_2)-B(\varphi_1)\right)\Vert\,.
\label{est24}
\end{align}
Here, and in the following, $\mathbb{Q}$ stands for a function \red{having similar properties as the
function $\mathbb{Q}_3$ from the statement of the theorem.}
Inserting the estimates \eqref{est23}--\eqref{est24} in \eqref{diffid6},
and choosing $\epsilon>0$ small enough, we
obtain \red{that almost everywhere in $(0,T)$ it holds}
\begin{align}
&\frac{d\Psi}{dt}+\red{\alpha_0}\,\Vert\varphi_t\Vert^2 \,\leq\, 4\,\epsilon'\,\Vert\nabla\uvec\Vert^2
+\Lambda_2\left(\Vert\varphi\Vert_V^2+\Vert\nabla\left(B(\varphi_2)-B(\varphi_1)\right)\Vert^2\right)
+\mathbb{Q}\,\Vert\uvec\Vert^2,\label{diffid7}
\end{align}
where
\begin{align}
&\Lambda_2:= C\left(1+\Vert\uvec_1\Vert_{H^2(\Omega)^2}^2+\Vert\varphi_{1,t}\Vert_V^2+\Vert\varphi_{2,t}\Vert_V^2\right)\in L^1(0,T).
\end{align}
We now \red{aim} to control the $L^2(\Omega)$ norm of \,$\nabla\left(B(\varphi_2)-B(\varphi_1)\right)$\,
by the $H^1(\Omega)$ norm of $\varphi$ (from above and below).
Now observe that
\begin{align*}
&\nabla\left(B(\varphi_2)-B(\varphi_1)\right)=\left(\lambda(\varphi_2)-\lambda(\varphi_1)\right)\nabla\varphi_2
+\lambda(\varphi_1)\nabla\varphi\,.
\end{align*}
Hence, we deduce that
\begin{align}
\Vert\nabla\left(B(\varphi_2)-B(\varphi_1)\right)\Vert^2
&\,\geq\, \red{\alpha_0^2}\,\Vert\nabla\varphi\Vert^2-
2\,\|\lambda\|_{C^0([-1,1])}\,\Vert\lambda(\varphi_2)-\lambda(\varphi_1)\Vert_{L^4(\Omega)}\,
\Vert\nabla\varphi_2\Vert_{L^4(\Omega)^2}\,\Vert\nabla\varphi\Vert\nonumber\\
&\,\geq\, \red{\alpha_0^2}\,\Vert\nabla\varphi\Vert^2-2\red{\,C\,} C_0\,\Vert\varphi\Vert_{L^4(\Omega)}\,\Vert\nabla\varphi_2\Vert_{L^4(\Omega)^2}\,
\Vert\nabla\varphi\Vert\nonumber\\
&\,\geq\,\red{\alpha_0^2}\,\Vert\nabla\varphi\Vert^2-\mathbb{Q}\left(\Vert\varphi\Vert+\Vert\varphi\Vert^{1/2}\,\Vert\nabla\varphi\Vert^{1/2}\right)
\Vert\nabla\varphi\Vert \nonumber\\
&\,\geq\,\frac{1}{2}\,\red{\alpha_0^2}\,\Vert\nabla\varphi\Vert^2-\mathbb{Q}\,\Vert\varphi\Vert^2\,.\label{est25}
\end{align}
On the other hand, it is immediately seen that we also have
\begin{align}
&\Vert\nabla\left(B(\varphi_2)-B(\varphi_1)\right)\Vert^2\,\leq\, C\,\Vert\varphi\Vert_V^2\,.\label{est26}
\end{align}
Thanks to \eqref{est25}, \eqref{est26}, and to the definition \eqref{Psi}, we then easily \red{find that}
\begin{align}
&\red{\frac{\alpha_0^2}4}\,\Vert\nabla\varphi\Vert^2-\mathbb{Q}\,\Vert\varphi\Vert^2
\,\leq\,\Psi\,\leq\, C\,\Vert\varphi\Vert_V^2.\label{est27}
\end{align}
Adding \eqref{est19} and \eqref{diffid7}, choosing $\epsilon'$ small enough, and
employing the bound \eqref{est27}, we are thus led to the differential inequality (cf. also \eqref{bound1})
\begin{align*}
&\frac{d}{dt}\Big(\Psi+\frac{1}{2}\Vert\uvec\Vert^2\Big)+\frac{\nu_1}{8}\Vert\nabla\uvec\Vert^2+
\red{\alpha_0}\,\Vert\varphi_t\Vert^2\notag\\
&\leq \Lambda_2
\left(\Psi+\frac{1}{2}\Vert\uvec\Vert^2\right)+ (\Lambda_2 + \mathbb{Q})\Vert\varphi\Vert^2+\frac{1}{\nu_1}\Vert\vvec\Vert_{V_{div}^\prime}^2,
\end{align*}
where $\vvec:=\vvec_2-\vvec_1$.
Hence, Gronwall's lemma, \eqref{stabi1}, and \eqref{est27} yield the stability estimate (cf. also \eqref{bound1})
\begin{equation}
\Vert\uvec\Vert_{L^\infty\left(0,t;G_{div}\right)}^2+\Vert\uvec\Vert_{L^2\left(0,t;V_{div}\right)}^2
+\Vert\varphi\Vert_{L^\infty\left(0,t;V\right)}^2+\Vert\varphi_t\Vert_{L^2\left(0,t;H\right)}^2
\,\leq\,\mathbb{Q}\,\Vert\vvec\Vert_{L^2(0,T;V_{div}^\prime)}^2\,.
\label{stabest1}
\end{equation}

\vspace{2mm}
We now aim to control the $L^2(0,t;H^2(\Omega))$ norm of $\varphi$ in terms of the $L^2(0,t;H)$ norm of $\varphi_t$.
This will be achieved in three steps.\vspace{2mm}

\noindent{\itshape Step 1. Control of \,\red{ $\|\Delta\big(B\left(\varphi_2\right)-B\left(\varphi_1\right)\big)\|_{
L^2(0,t;H)}$
\,in terms of \,$\|\varphi_t\|_{L^2(0,t;H)}$.}}

\vspace{2mm}\noindent
\red{We write \eqref{stt3} for both solutions and take the difference of the equations. We then get
the identity}
\newpage
\begin{align}
\Delta\big(B\left(\varphi_2\right)-B\left(\varphi_1\right)\big)
&\,=\,\varphi_t+\uvec\cdot\nabla\varphi_2+\uvec_1\cdot\nabla\varphi
+\big(m(\varphi_2)-m(\varphi_1)\big)\,\mbox{div}(\nabla K\ast\varphi_2)\nonumber\\
&\quad\,\,+\big(\left(m'(\varphi_2)-m'(\varphi_1)\right)\nabla\varphi_2+m'(\varphi_1)\nabla\varphi\big)
\cdot(\nabla K\ast\varphi_2)\nonumber\\
&\quad\,\, +m(\varphi_1)\,\mbox{div}(\nabla K\ast\varphi)+m'(\varphi_1)\nabla\varphi_1\cdot
\left(\nabla K\ast\varphi\right).\label{eqdiff}
\end{align}
It is easy to see that \red{the $L^2(\Omega)$ norms of the fourth to last terms on the right-hand side} of \eqref{eqdiff}
can, on account of Lemma \ref{admiss} and of the bound \eqref{reg3}$_1$
for $\varphi_1,\varphi_2$, be estimated by $C\,\Vert\varphi\Vert_V$. \red{By virtue of Poincar\'e's
inequality, we therefore} get that
\begin{align}
\Vert\Delta\big(B\left(\varphi_2\right)-B\left(\varphi_1\right)\big)\Vert
&\leq
\Vert\varphi_t\Vert+\Vert\uvec\Vert_{L^4(\Omega)^2}\Vert\nabla\varphi_2\Vert_{L^4(\Omega)^2}
+\Vert\uvec_1\Vert_{L^4(\Omega)^2}\Vert\nabla\varphi\Vert_{L^4(\Omega)^2}
+C\Vert\varphi\Vert_V\nonumber\\
&\leq \Vert\varphi_t\Vert+C\Vert\nabla\uvec\Vert+C\Vert\nabla\varphi\Vert^{1/2}\Vert\varphi\Vert_{H^2(\Omega)}^{1/2}
+C\Vert\varphi\Vert_V\nonumber\\
&\leq\Vert\varphi_t\Vert+C\Vert\nabla\uvec\Vert+\delta\Vert\varphi\Vert_{H^2(\Omega)}+C_\delta\Vert\varphi\Vert_V,
\label{est33}
\end{align}
for every $\delta>0$ (to be fixed later).

\vspace{3mm}\noindent
{\itshape Step 2. Control of \red{\, $\|B\left(\varphi_2\right)-B\left(\varphi_1\right)\|_{L^2(0,t;H^2(\Omega))}$
in terms of \,$\|\Delta\big(B\left(\varphi_2\right)-B\left(\varphi_1\right)\big)\|_{L^2(0,t;H)}$}.
}\\
We need to estimate the trace of the normal derivative of $B\left(\varphi_2\right)-B\left(\varphi_1\right)$
in $H^{1/2}(\partial\Omega)$.
For this purpose, we write \eqref{stt5}$_2$  for each solution and then take the difference. From the
resulting equation, we get that
\begin{equation*}
\frac{\partial}{\partial\nvec}\big(B\left(\varphi_2\right)-B\left(\varphi_1\right)\big)
=\left(m(\varphi_2)-m(\varphi_1)\right)\left(\nabla K\ast\varphi_2\right)\cdot\nvec
+m(\varphi_1)\left(\nabla K\ast\varphi\right)\cdot\nvec\quad\mbox{ a.e. on }\Sigma\,.
\end{equation*}
By applying Lemma \ref{trace-product}, we then obtain the estimate
\begin{align}
&\Big\Vert\frac{\partial}{\partial\nvec}\big(B\left(\varphi_2\right)-B\left(\varphi_1\right)\big)
\Big\Vert_{H^{1/2}(\partial\Omega)}
\,\leq\,\Vert m(\varphi_2)-m(\varphi_1)\Vert_{L^\infty(\partial\Omega)}\,
\Vert\left(\nabla K\ast\varphi_2\right)\cdot\nvec\Vert_{H^{1/2}(\partial\Omega)}\nonumber\\
&\quad +\Vert\left(\nabla K\ast\varphi_2\right)\cdot\nvec \Vert_{L^\infty(\partial\Omega)}\,
\Vert m(\varphi_2)-m(\varphi_1)\Vert_{H^{1/2}(\partial\Omega)}\nonumber\\
&\quad +\Vert m(\varphi_1)\Vert_{L^\infty(\partial\Omega)}\,
\Vert\left(\nabla K\ast\varphi\right)\cdot\nvec\Vert_{H^{1/2}(\partial\Omega)}\,+\,
\Vert\left(\nabla K\ast\varphi\right)\cdot\nvec \Vert_{L^\infty(\partial\Omega)}\,
\Vert m(\varphi_1)\Vert_{H^{1/2}(\partial\Omega)}\nonumber\\
&\quad =: \sum_{j=1}^4 I^{(2)}_j \,,\label{est30}
\end{align}
\red{with obvious meaning of $I_j^{(2)}$, $1\le j\le 4$.}
We now proceed to estimate the four terms on the right-hand side \red{individually. To this end,
we employ Lemma 1, Agmon's inequality \eqref{Agmon}, and the classical trace theorem, where $C_{tr}$
denotes the constant of the continuous embedding $H^1(\Omega)\subset H^{1/2}(\partial\Omega)$.
We also utilize the fact that if $\psi\in H^1(\Omega)$ and $|\psi|\leq\zeta$ almost everywhere in $\Omega$
 for some positive constant $\zeta$
(with $\Omega$ smooth enough), then the trace $\gamma_0\psi:=\psi|_{\partial\Omega}\in H^{1/2}(\partial\Omega)$ of $\psi$ on the boundary $\partial\Omega$ satisfies $|\gamma_0\psi|\leq\zeta$ a.e. on $\partial\Omega$, and, moreover, if
$\blue{g}\in C^1(\mathbb{R})$, then $\blue{g}(\psi)\in H^1(\Omega)$ and $\gamma_0 \blue{g}(\psi)=\blue{g}(\gamma_0\psi)$.
With these tools at hand, we deduce, for every $\delta>0$ (to be fixed later), the  chain of estimates}
\newpage
\begin{align}
I^{(2)}_1
&\,\leq\, m_\infty '\,\Vert\varphi\Vert_{L^\infty(\Omega)}\,\Vert K\ast\varphi_2\Vert_{H^2(\Omega)}
\,\leq\, C_{m,K,\Omega}\,\Vert\varphi\Vert^{1/2}\,\Vert\varphi\Vert_{H^2(\Omega)}^{1/2}\nonumber\\
&\,\leq\,\delta\,\Vert\varphi\Vert_{H^2(\Omega)}+C_{\delta,m,K,\Omega}\,\Vert\varphi\Vert\,,\label{est28}\\[1mm]
I^{(2)}_2
&\,\leq\, \Vert\left(\nabla K\ast\varphi_2\right)\cdot\nvec \Vert_{L^\infty(\partial\Omega)}\,C_{tr}\,\Vert m(\varphi_2)-m(\varphi_1)\Vert_{V}\,\leq\, C_{m,K,\Omega}\,\Vert\varphi\Vert_V\,,\\[1mm]
I^{(2)}_3
&\,\leq\, m_\infty\,\Vert K\ast\varphi\Vert_{H^2(\Omega)}\,\leq \,C_{m,K,\Omega}\,\Vert\varphi\Vert,\\[1mm]
I^{(2)}_4
&\,\leq\,\Vert\nabla K\ast\varphi\Vert_{L^\infty(\Omega)^2}\,C_{tr}\,\Vert m(\varphi_1)\Vert_V
\,\le\, C_{m,K,\Omega}\,\Vert\varphi\Vert_{L^\infty(\Omega)}\nonumber\\
&\,\leq\,\delta\,\Vert\varphi\Vert_{H^2(\Omega)}+C_{\delta,m,K,\Omega}\,\Vert\varphi\Vert\,.\label{est29}
\end{align}
Inserting the estimates \eqref{est28}--\eqref{est29} in \eqref{est30},
and invoking \eqref{est26}, we deduce that
\begin{align}
&\Vert B(\varphi_2)-B(\varphi_1)\Vert_{H^2(\Omega)}
\,\leq\, C\,\Vert\Delta\left(B(\varphi_2)-B(\varphi_1)\right)\Vert+C\,\delta\,\Vert\varphi\Vert_{H^2(\Omega)}+
C_\delta\,\Vert\varphi\Vert_V \red{\quad\mbox{a.e. in \,$(0,T)$}}.
\label{est34}
\end{align}
{\itshape Step 3. Control of \,\red{$\|\varphi\|_{L^2(0,t;H^2(\Omega))}$\, in terms of\,
$\|B(\varphi_2)-B(\varphi_1)\|_{L^2(0,t;H^2(\Omega))}$}.}

\vspace{2mm}\noindent
We write the identity (cf. \eqref{defB}) \,$\partial_j\varphi=\lambda^{-1}\partial_j B\red{(\varphi)}$, $j=1,2$,
for the two solutions and take the difference. For the second spatial derivatives $\partial^2_{ij}\varphi$, we get
\begin{align}
\partial_{ij}^2\varphi&=\frac{1}{\lambda(\varphi_1)}\partial_{ij}^2\big(B(\varphi_2)-B(\varphi_1)\big)
+\Big(\frac{1}{\lambda(\varphi_2)}-\frac{1}{\lambda(\varphi_1)}\Big)\partial_{ij}^2 B(\varphi_2)\nonumber\\
&\quad -\Big(\frac{1}{\lambda^2(\varphi_2)}-\frac{1}{\lambda^2(\varphi_1)}\Big)\partial_i\lambda(\varphi_2)\partial_j B(\varphi_2)
-\frac{1}{\lambda^2(\varphi_1)}\big(\partial_i\lambda(\varphi_2)-\partial_i\lambda(\varphi_1)\big)\partial_j B(\varphi_2)
\nonumber\\
&\quad -\frac{1}{\lambda^2(\varphi_1)}\partial_i\lambda(\varphi_1)\big(\partial_j B(\varphi_2)-\partial_j B(\varphi_1)\big)
-\frac{1}{\lambda(\varphi_1)}\big( m(\varphi_2)-m(\varphi_1)\big)\partial_i\varphi_2\,.\label{id4}
\end{align}
Let us denote by $I^{(3)}_j$, $j
=1,\dots,6$, the $L^2(\Omega)$ norms of the six terms on the right-hand side of the above identity.
\red{Now observe that \eqref{reg3} implies that $\partial_i\varphi_1,\partial_i\varphi_2
\in L^{\infty}(0,T;L^p(\Omega))$ \,for $i=1,2$ and all \,$p\in [1,+\infty)$. We can therefore infer
from ${\bf (H2^{**})}$, \eqref{GN}, and Young's inequality the estimate}
\begin{align}
\Vert\partial_i\lambda(\varphi_2)-\partial_i\lambda(\varphi_1)\Vert_{L^4(\Omega)}
&\leq C\Vert\varphi\Vert_V+C\Vert\nabla\varphi\Vert_{L^4(\Omega)^2}\nonumber\\
&\leq C\Vert\varphi\Vert_V+C\Vert\nabla\varphi\Vert^{1/2}\Vert\varphi\Vert_{H^2(\Omega)}^{1/2}\nonumber\\
&\leq \eta\Vert\varphi\Vert_{H^2(\Omega)}+C_\eta\Vert\varphi\Vert_V,
\end{align}
for any  $\eta>0$ (to be chosen later). The terms $I^{(3)}_k$ can be estimated in the following way:
\begin{align}
I^{(3)}_1&\,\leq\,\frac{1}{\red{\alpha_0}}\,\Vert\partial_{ij}^2\big(B(\varphi_2)-B(\varphi_1)\big)\Vert\,,\label{est31}
\\[1mm]
I^{(3)}_2&\,\leq \,C\,\Vert\varphi\Vert_{L^\infty(\Omega)}\,\Vert\partial_{ij}B(\varphi_2)\Vert
\,\leq\, C\,\Vert\varphi\Vert^{1/2}\,\Vert\varphi\Vert_{H^2(\Omega)}^{1/2}\nonumber\\
&\,\leq\,\eta\,\Vert\varphi\Vert_{H^2(\Omega)}+C_\eta\,\Vert\varphi\Vert\,,\\[1mm]
I^{(3)}_3&\,\leq\, C\,\Vert\varphi\Vert_{L^\infty(\Omega)}\,\Vert\partial_i\lambda(\varphi_2)\Vert_{L^4(\Omega)}\,
\Vert\partial_j B(\varphi_2)\Vert_{L^4(\Omega)}\nonumber\\
&\,\leq C\,\Vert\varphi\Vert^{1/2}\,\Vert\varphi\Vert_{H^2(\Omega)}^{1/2}
\,\leq\,\eta\,\Vert\varphi\Vert_{H^2(\Omega)}+C_\eta\Vert\varphi\Vert\,,\\[1mm]
I^{(3)}_4&\,\leq\, C\,\Vert\partial_i\lambda(\varphi_2)-\partial_i\lambda(\varphi_1)\Vert_{L^4(\Omega)}\,
\Vert\partial_j B(\varphi_2)\Vert_{L^4(\Omega)}\nonumber\\
&\,\leq\, C\,\eta\,\Vert\varphi\Vert_{H^2(\Omega)}+C_\eta\,\Vert\varphi\Vert_V\,,\\[1mm]
I^{(3)}_5&\,\leq \,C\,\Vert\partial_i\lambda(\varphi_1)\Vert_{L^4(\Omega)}\,
\Vert\partial_j B(\varphi_2)-\partial_j B(\varphi_1)\Vert_{L^4(\Omega)}\nonumber\\
&\,\leq\, C\,\Vert B(\varphi_2)- B(\varphi_1)\Vert_{H^2(\Omega)}\,,\\[1mm]
I^{(3)}_6&\,\leq\, C\,\Vert\varphi\Vert_V\,.\label{est32}
\end{align}
\red{Here, we have used Agmon's inequality \eqref{Agmon}, as well as} the fact that
\,$B(\varphi_2)\in L^\infty(0,T;H^2(\Omega))$\,
and \,$\lambda(\varphi_j)\in L^\infty(0,T;W^{1,p}(\Omega))$, for all $1\le p<+\infty$, $j=1,2$.
By means of the estimates \eqref{est31}--\eqref{est32}, and taking $\eta>0$ small enough, we deduce from \eqref{id4} that
\begin{align}
&\Vert\varphi\Vert_{H^2(\Omega)}\,\leq\, C
\,\Vert B(\cdot,\varphi_2)- B(\cdot,\varphi_1)\Vert_{H^2(\Omega)}+C\,\Vert\varphi\Vert\,.
\label{est35}
\end{align}
Now, combining the estimates \eqref{est33}, \eqref{est34}, \eqref{est35} obtained in the three steps above,
and fixing $\delta>0$ small enough, we finally deduce the desired control
\begin{align}
&\Vert\varphi\Vert_{H^2(\Omega)}\,\leq\, C\,\Vert\varphi_t\Vert+C\,\Vert\nabla\uvec\Vert+C\,\Vert\varphi\Vert_V\,.
\label{est36}
\end{align}
The stability estimate \eqref{stabest2} now immediately follows from \eqref{stabest1} and \eqref{est36}.
This concludes the proof of the lemma.
\end{proof}

\section{Optimal control}
\label{optimal control}
\setcounter{equation}{0}

We now study the optimal control problem \textbf{(CP)}. Throughout this section, we assume that the cost functional $\,{\cal J}\,$ is given by
(\ref{costfunct}).  Moreover, we assume that the set of admissible controls $\mathcal{V}_{ad}$ is \blue{defined} by
\begin{align}
\label{Vad}
&\mathcal{V}_{ad}:=\big\{\vvec\in L^2(0,T;G_{div}):\:\: v_{a,i}(x,t)\leq v_i(x,t)\leq v_{b,i}(x,t),\:\:\mbox{a.e. }(x,t)\in Q,\:\: i=1,2\big\},
\end{align}
with given functions $\vvec_a,\vvec_b\in L^2(0,T;G_{div})\cap L^\infty(Q)^2$.
Notice that the stability estimate provided by Lemma \ref{stablem2} yields that the control-to-state map $\,{\cal S}\,$
introduced above (cf. \eqref{controlmap1}, \eqref{controlmap2}) is locally
Lipschitz continuous from ${\cal V}$ into the space
\begin{align}
&
\mathcal{W}:=\big[ { C^0([0,T];G_{div})}
\cap L^2(0,T;V_{div})\big]\times\big[H^1(0,T;H)\cap C^0([0,T];V)\cap L^2(0,T;H^2(\Omega))\big].
\end{align}
We also point out that problem \textbf{(CP)} can be formulated \red{in the form}
\begin{align}
&\min_{\vvec\in\mathcal{V}_{ad}} f(\vvec),\nonumber
\end{align}
for the reduced cost functional defined by $f(\vvec):=\mathcal{J}\big({\cal S}(\vvec),\vvec\big)$, for every $\vvec\in\mathcal{V}$.

Let us first prove that an optimal control exists.
\begin{thm}
Let the assumptions of Lemma \ref{stablem2} hold true. Then the optimal control problem {\bf (CP)} on $\mathcal{V}_{ad}$ admits a solution.
\end{thm}
\begin{proof}
In the first part of the proof, we can argue as in \cite[Proof of Theorem 2]{FRS}. We pick a minimizing sequence $\{\vvec_n\}\subset\mathcal{V}_{ad}$ for \textbf{(CP)}, and since
${\cal V}_{ad}$ is
bounded in ${\cal V}$, we may assume without loss of generality that
$\vvec_n\to\overline{\vvec}$ weakly in $L^2(0,T;G_{div})$
for some $\overline{\vvec}\in {\cal V}$. Since $\mathcal{V}_{ad}$ is convex and closed in $\mathcal{V}$,
and thus weakly sequentially closed, we have that $\overline{\vvec}\in\mathcal{V}_{ad}$.

Moreover, ${\cal S}$ is a locally bounded mapping from ${\cal V}$ into ${\cal H}$. Hence,
setting $\,[\uvec_n,\varphi_n]:={\cal S}(\vvec_n)$,
$n\in \mathbb{N}$, we may without loss of generality assume that, with appropriate limit points $[\overline{\uvec},\overline{\varphi}]$,
\begin{align}
&
\uvec_n\to\overline{\uvec}\,\quad\mbox{weakly$^\ast$ in $L^\infty(0,T;V_{div})$,\, weakly in $H^1(0,T;G_{div})\cap L^2(0,T;H^2(\Omega)^2)$},
\label{wconv1}\\
&\varphi_n\to\overline{\varphi}\,\quad\mbox{weakly$^\ast$ in $L^\infty(0,T;H^2(\Omega))\cap W^{1,\infty}(0,T;H)$,\, weakly in $H^1(0,T;V)$}.
\label{wconv2}
\end{align}
In particular, it follows from the compactness of the embedding $H^1(0,T;V)\cap L^\infty(0,T;H^2(\Omega))\linebreak
\subset C^0([0,T];{H^r(\Omega))}$ for $0\le r<2$, given by the Aubin-Lions lemma (cf. \cite{Lions}), that $\,\varphi_n\to \overline{\varphi}$ strongly in $C^0(\overline{Q})$. Hence, we have $\nu(\varphi_n)\to \nu(\overline{\varphi})$ strongly in $C^0(\overline{Q})$.
Moreover, we also have, by compact embedding, that
$\uvec_{n}\to\overline{\uvec}$ strongly in $L^2(0,T;G_{div})$.
By employing these weak and strong convergence properties, we can now pass to the limit in the weak formulation of
\red{the state system}
\eqref{sy3}--\eqref{sy6} (cf. Definition \ref{wfdef})
to see that $[\overline{\uvec},\overline{\varphi}]$ satisfies the weak formulation corresponding to $\overline{\vvec}$.
Notice that, instead of passing to the limit in the weak formulation of the nonlocal Cahn--Hilliard equation \eqref{sy1}
given in Definition \ref{wfdef}, we can alternatively pass to the limit in the weak formulation of
\eqref{stt3}, which reads
\begin{align}
&\langle\varphi_t,\psi\rangle\red{_V}+\left(\nabla B(\cdot,\varphi),\nabla\psi\right)
-\left(m(\varphi)\left(\nabla K\ast\varphi\right),\nabla\psi\right)=\left(\uvec\varphi,\nabla\psi\right)
\end{align}
for every $\psi\in V$ \blue{and almost any $t\in(0,,T)$}.
Hence, we have
$[\overline{\uvec},\overline{\varphi}]={\cal S}(\overline{\vvec})$, that is, the pair $([\overline{\uvec},\overline{\varphi}],
\overline{\vvec})$
is admissible for \textbf{(CP)}.
Finally, thanks to the weak sequential lower semicontinuity of $\mathcal{J}$ and to the weak convergences
\eqref{wconv1}, \eqref{wconv2}, we infer that
the state $[\overline{\uvec},\overline{\varphi}]={\cal S}(\overline{\vvec})$
is a solution to \textbf{(CP)}.
\end{proof}

\noindent
\textbf{The linearized system}.  Assume that the assumptions of Lemma \ref{stablem2} are fulfilled.
We fix a control $\overline{\vvec}\in\mathcal{V}$ and let
$[\overline{\uvec},\overline{\varphi}]:=\mathcal{S}(\overline{\vvec})\in\mathcal{H}$ be the
associated unique strong solution to the state system \eqref{sy3}--\red{\eqref{sy6}} according to Theorem \ref{reg-thm}.
Let an arbitrary $\hvec\in\mathcal{V}$ be given. In order to prove Fr\'echet differentiability of
the control-to-state operator at $\overline{\vvec}$,
we first consider the following system, which is obtained by linearizing the state system \eqref{sy3}--\red{\eqref{sy6}}
at $[\overline{\uvec},\overline{\varphi}]$:
\begin{align}
&\xivec_t-2\,\mbox{div}\big(\nu(\overline{\varphi})D\xivec\big)-2\,\mbox{div}\big(\nu^{\,\prime}(\overline{\varphi})\,
\eta\, D\overline{\uvec}\big)
+\left(\overline{\uvec}\cdot\nabla\right)\xivec+\left(\xivec\cdot\nabla\right)\overline{\uvec}+\nabla\pi^\ast\nonumber\\
&=\,\eta\left(\nabla K\ast\overline{\varphi}\right)+\overline{\varphi}\left(\nabla K\ast\eta\right)+\hvec\,
\quad\mbox{ in }Q,\label{linsy1}\\
&\eta_t+\overline{\uvec}\cdot\nabla\eta\,=\,-\xivec\cdot\nabla\overline{\varphi}+\mbox{div}\big(\lambda(\overline{\varphi})\nabla\eta\big)
-\mbox{div}\big(m'(\overline{\varphi})\eta\nabla\left(K\ast\overline{\varphi}\right)\big)\nonumber\\
&-\mbox{div}\big(m(\overline{\varphi})\left(\nabla K\ast\eta\right)\big)
+\mbox{div}\big(\eta\lambda^\prime(\overline{\varphi})\nabla\overline{\varphi}\big)\qquad\mbox{in }Q,\label{linsy2}\\
&\mbox{div}(\xivec)=0,\quad\mbox{ in }Q,\\
&\xivec=\mathbf{0}\,\quad\mbox{ on }\Sigma,\\
&\big[\lambda(\overline{\varphi})\nabla\eta-m'(\overline{\varphi})\eta\nabla\big(K\ast\overline{\varphi}\big)
-m(\overline{\varphi})\big(\nabla K\ast\eta\big)+\eta\lambda'(\overline{\varphi})\nabla\overline{\varphi}\big]\cdot\nvec=0\,
\quad\mbox{on }\Sigma,\label{linsybc}\\
&\xivec(0)=\mathbf{0},\qquad\eta(0)=0\,\quad\mbox{ in }\Omega.\label{linics}
\end{align}

We first prove that system \eqref{linsy1}--\eqref{linics} has a unique weak solution.
\begin{prop}\label{linthm}
Let the assumptions of Lemma \ref{stablem2} be satisfied.
Then problem \eqref{linsy1}--\eqref{linics}
has for every $\hvec\in\mathcal{V}$
a unique weak solution $\,[\xivec,\eta]\,$ such that
\begin{align}
&\xivec\in
H^1(0,T;V_{div}^\prime)
\cap C^0([0,T];G_{div})\cap L^2(0,T;V_{div}),
\nonumber\\
&\eta\in H^1(0,T;V')\cap C^0([0,T];H)\cap L^2(0,T;V).\label{reglin}
\end{align}
\red{Moreover, there is some constant $K_1^*>0$, which depends only on the data of the state system,
such that, for every $t\in (0,T]$,}
\begin{align}\label{stabulin}\red{
\|\xivec\|_{H^1(0,t;V_{div}')\cap C^0([0,t];G_{div})\cap L^2(0,t;V_{div})}+\|\eta\|_{H^1(0,t;V')
\cap C^0([0,t];H)\cap L^2(0,t;V)}\,\le\,K_1^*\,\|\hvec\|_{\cal V}\,.}
\end{align}
\end{prop}

\begin{proof}
We make use of a Faedo--Galerkin approximating scheme. Following the lines of \cite{CFG},
we introduce the family $\{\wvec_j\}_{j\in\mathbb{N}}$ of the eigenfunctions to the Stokes operator $S$ as a
Galerkin basis in $V_{div}$
and the family $\{\psi_j\}_{j\in\mathbb{N}}$ of the eigenfunctions to
 the Neumann operator $A:=-\Delta+I$ as a Galerkin basis in $V$.
 Both these eigenfunction families $\{\wvec_j\}_{j\in\mathbb{N}}$ and
$\{\psi_j\}_{j\in\mathbb{N}}$ are assumed to be suitably ordered and normalized.
Moreover, recall that, since $\wvec_j\in D(S)$, we have ${\rm div}(\wvec_j)=0$.

Then we look for two functions of the form
\begin{align}
&\xivec_n(t):=\sum_{j=1}^n a^{(n)}_j(t)\wvec_j\,,\qquad\eta_n(t):=\sum_{j=1}^n
b^{(n)}_j(t) \psi_j\,,\nonumber
\end{align}
that solve the following approximating problem:
\begin{align}
&\langle \partial_t\xivec_n(t),\wvec_i\rangle_{V_{div}}\,+\,2\,\big(\nu(\overline{\varphi}(t))
\,D\xivec_n(t),D\wvec_i\big)
\,+\,
2\,\big(\nu'(\overline{\varphi}(t))\,\eta_n(t)\, D\overline{\uvec}(t),D\wvec_i\big)
\nonumber\\
&+\,b(\overline{\uvec}(t),\xivec_n(t),\wvec_i)\,+\,b(\xivec_n(t),\overline{\uvec}(t),\wvec_i)\nonumber\\
&\,=\,\big(\eta_n(t)(\nabla K\ast\overline{\varphi})(t),\wvec_i\big)
\,+\,\big(\overline{\varphi}(t)(\nabla K\ast\eta_n)(t),\wvec_i\big)\,+\,
(\hvec(t),\wvec_i)
\,,\label{FaGa1}\\[1mm]
&\langle \partial_t\eta_{n}(t),\psi_i\rangle_V
\,=\,-\big(\lambda(\overline{\varphi}(t))\nabla\eta_n(t),\nabla\psi_i\big)\,+\,
\big(m'(\overline{\varphi}(t))\eta_n(t)\nabla\left(K\ast\overline{\varphi}\right)(t),\nabla\psi_i\big)\nonumber\\
&\,+\,\big(m(\overline{\varphi}(t))\left(\nabla K\ast\eta_n\right)(t),\nabla\psi_i\big)
\,-\,\big(\eta_n(t)\lambda^\prime(\overline{\varphi}(t))\nabla\overline{\varphi}(t),\nabla\psi_i\big)
+(\overline{\uvec}(t)\,\eta_n(t),\nabla\psi_i)\nonumber\\
&
+(\xivec_n(t)\, \overline{\varphi}(t),\nabla\psi_i),\label{FaGa2}\\
&\xivec_n(0)= \mathbf{0},
\,\quad\eta_n(0)=0,\label{FaGa3}
\end{align}
for $i=1,\dots,n$, and for almost every $t\in (0,T)$. It is \red{immediately seen} that the above system
can be reduced to a Cauchy problem for a system of $2n$ linear ordinary differential equations in the $2n$
unknowns $a^{(n)}_i$, $b^{(n)}_i$,
in which, owing to the regularity properties of $[\overline{\uvec},\overline{\varphi}]$, all of
the coefficient functions belong to $\,L^2(0,T)$.
Thanks to Carath\'{e}odory's theorem,
we can conclude that this problem enjoys
a unique solution $\boldsymbol{a}^{(n)}:=(a^{(n)}_1,\cdots,a^{(n)}_n)^{\mathbf{t}}$,
$\boldsymbol{b}^{(n)}:=(b^{(n)}_1,\cdots,b^{(n)}_n)^{\mathbf{t}}$,
such that $\boldsymbol{a}^{(n)},\boldsymbol{b}^{(n)}\in H^1(0,T;\mathbb{R}^n)$, which
then specifies $[\xivec_n,\eta_n]$.

We now derive a priori estimates for $\xivec_n$ and $\eta_n$ that are uniform in $n\in\mathbb{N}$.
To begin with, let us
multiply \eqref{FaGa1} by $a^{(n)}_i$, \eqref{FaGa2} by $b^{(n)}_i$,
sum over $i=1,\cdots,n$, and add the resulting identities.  Omitting the argument $t$ for the sake of a shorter exposition,
we then obtain, almost
everywhere in $(0,T)$,
\begin{align}
&\frac{1}{2}\,\frac{d}{dt}\,\big(\Vert\xivec_n\Vert^2+\Vert\eta_n\Vert^2\big)\,+\,
2\,\big(\nu(\overline{\varphi})\,D\xivec_n,D\xivec_n\big)\,+\,
\big(\lambda(\overline{\varphi})\nabla\eta_n,\nabla\eta_n\big)=\,-b(\xivec_n,\overline{\uvec},\xivec_n) \nonumber
\\
&\,-\,2\,\big(\nu^{\,\prime}(\overline{\varphi})\,\eta_n\,
D\overline{\uvec},D\xivec_n\big)\,+\,\big(\eta_n(\nabla K\ast\overline{\varphi}),\xivec_n\big)
\,+\,\big(\overline{\varphi}\,(\nabla K\ast\eta_n),\xivec_n\big)\nonumber\\
&\,+\,
(\hvec,\xivec_n)\,+\,
\big(m'(\overline{\varphi})\eta_n\nabla\left(K\ast\overline{\varphi}\right),\nabla\eta_n\big)
\,+\,\big(m(\overline{\varphi})\left(\nabla K\ast\eta_n\right),\nabla\eta_n\big)\nonumber\\
&\,-\,\big(\eta_n\lambda^\prime(\overline{\varphi})\nabla\overline{\varphi},\nabla\eta_n\big)\,
+\,(\xivec_n\,\overline{\varphi},\nabla\eta_n).\label{diffid9}
\end{align}
Let us now estimate the terms on the right-hand side of this equation individually.
In the remainder of this proof, we use the following abbreviating notation:
the letter $\,C\,$ will stand for positive constants that depend only on the global data of  the \red{state system}, on $\overline{\vvec}$,  and on
$[\overline{\uvec},\overline{\varphi}]$, but not on $n\in\mathbb{N}$;
moreover, by $C_\sigma$ we denote constants that in addition depend on the quantities
indicated
by the index $\sigma$, but not on $n\in\mathbb{N}$. Both $C$ and $C_\sigma$ may change
 within formulas and even
within lines.

The first two terms on the right-hand side can be estimated exactly as in \cite[Proof of Proposition 1]{FRS}, namely,
\begin{align}
&
|b(\xivec_n,\overline{\uvec},\xivec_n)|
\,\leq\,\epsilon\,\Vert\nabla\xivec_n\Vert^2+C_\epsilon\,\Vert\overline{\uvec}\Vert_{H^2(\Omega)^2}^2\,\Vert\xivec_n\Vert^2,
\label{est41}\\[2mm]
&
\big|2\,\big(\nu^{\,\prime}(\overline{\varphi})\,\eta_n \,D\overline{\uvec},D\xivec_n\big)\big|\,
\leq\,\epsilon\,\Vert\nabla\xivec_n\Vert^2\,+\,\epsilon'\,\Vert\nabla\eta_n\Vert^2
\,+\,C_{\epsilon,\epsilon'}\,\Vert\overline{\uvec}\Vert_{H^2(\Omega)^2}^2\,\Vert\eta_n\Vert^2.\label{est42}
\end{align}
Concerning the other terms, we have, using H\"older's inequality, Young's inequality, and the global
bounds (\ref{bound1}) as main tools,
the following series of estimates:
\begin{align}
\big|\big(\eta_n(\nabla K\ast\overline{\varphi}),\xivec_n\big)\big|&\,\leq\,C_K(\Vert\eta_n\Vert^2+\Vert\xivec_n\Vert^2)\,,\label{est43}\\[2mm]
\big|\big(\overline{\varphi}\,(\nabla K\ast\eta_n),\xivec_n\big)\big|&\,\leq\,C_K(\Vert\eta_n\Vert^2+\Vert\xivec_n\Vert^2)\,, \label{est431}\\[2mm]
|(\hvec,\xivec_n)|&\,\leq\,\frac{1}{2}(\Vert\hvec\Vert^2+\Vert\xivec_n\Vert^2)\,,\label{est432}\\[2mm]
\big|\big(m'(\overline{\varphi})\eta_n\nabla\left(K\ast\overline{\varphi}\right),\nabla\eta_n\big)\big|
&\,\leq\,\Vert m'(\overline{\varphi})\Vert_{L^\infty(\Omega)}
\Vert\eta_n\Vert_{L^4(\Omega)}\Vert\nabla\left(K\ast\overline{\varphi}\right)\Vert_{L^4(\Omega)}\Vert\nabla\eta_n\Vert
\nonumber\\[1mm]
&\,\leq\,\epsilon'\Vert\nabla\eta_n\Vert^2
\,+\,C_{m,\epsilon'}\Vert\eta_n\Vert_{L^4(\Omega)}^2\Vert\nabla\left(K\ast\overline{\varphi}\right)\Vert_{L^4(\Omega)}^2
\nonumber\\[1mm]
&\,\leq\,\epsilon'\Vert\nabla\eta_n\Vert^2\,+\,C_{m,K,\epsilon'}\Vert\overline{\varphi}\Vert_{H^2(\Omega)}^2
\big(\Vert\eta_n\Vert^2+\Vert\eta_n\Vert\Vert\nabla\eta_n\Vert\big)\nonumber\\[1mm]
&\,\leq\,2\epsilon'\Vert\nabla\eta_n\Vert^2\,+\,C_{m,K,\epsilon'}\Vert\eta_n\Vert^2\,, \label{est433}\\[2mm]
\big|\big(m(\overline{\varphi})\left(\nabla K\ast\eta_n\right),\nabla\eta_n\big)\big|&\,\leq\,
\Vert m(\overline{\varphi})\Vert_{L^\infty(\Omega)}\Vert\nabla K\ast\eta_n\Vert\Vert\nabla\eta_n\Vert\nonumber\\[1mm]
&\,\leq\,\epsilon'\Vert\nabla\eta_n\Vert^2\,+\,C_{m,K,\epsilon'}\Vert\eta_n\Vert^2\,,\nonumber\\[2mm]
\big|\big(\eta_n\lambda^\prime(\overline{\varphi})\nabla\overline{\varphi},\nabla\eta_n\big)\big|&\,\leq\,
\Vert\eta_n\Vert_{L^4(\Omega)}\Vert\lambda^\prime(\overline{\varphi})\Vert_{L^\infty(\Omega)}
\Vert\nabla\overline{\varphi}\Vert_{L^4(\Omega)}\Vert\nabla\eta_n\Vert \label{est434}\\[1mm]
&\,\leq\,\epsilon'\Vert\nabla\eta_n\Vert^2\,+\,C_{\lambda,\epsilon'}\Vert\eta_n\Vert^2\,,\nonumber\\[2mm]
|(\xivec_n\,\overline{\varphi},\nabla\eta_n)|&\,\leq\,\epsilon'\Vert\nabla\eta_n\Vert^2\,+\,C_{\epsilon'}\Vert\xivec_n\Vert^2\,.\label{est44}
\end{align}
Hence, inserting the estimates \eqref{est41}-\eqref{est44} in \eqref{diffid9}
and choosing $\epsilon>0$ and $\epsilon'>0$ small enough (i.e., $\epsilon\leq \nu_1/4$
 and $\epsilon'\leq\red{\alpha_0}/{12}$), we obtain the estimate
\begin{align}
&\frac{d}{dt}\big(\Vert\xivec_n\Vert^2+\Vert\eta_n\Vert^2\big)+{{\nu_1}}\,\Vert\nabla\xivec_n\Vert^2\,+\,
\red{\alpha_0}
\,\Vert\nabla\eta_n\Vert^2\nonumber\\[1mm]
&\leq {{C}}\,\big(1+\Vert\overline{\uvec}\Vert_{H^2(\Omega)^2}^2\big)\big(\Vert\xivec_n\Vert^2
+\Vert\eta_n\Vert^2\big)\,+\,\Vert\hvec\Vert^2.\label{Gr}
\end{align}
Observe now that, owing to (\ref{bound1}), the mapping $\,t\mapsto \|\overline{\uvec}(t)\|_{H^2(\Omega)^2}^2\,$ belongs to
$L^1(0,T)$. Therefore, Gronwall's lemma yields, for every $t\in (0,T]$,
\begin{equation}
\Vert\xivec_n\Vert_{L^\infty(0,t;G_{div})\cap L^2(0,t;V_{div})}\,\leq\,
C\,\Vert\hvec\Vert_\mathcal{V},
\,\quad\Vert\eta_n\Vert_{L^\infty(0,t;H)\cap L^2(0,t;V)}\,\leq\,C\,\Vert\hvec\Vert_\mathcal{V}
\,,
\label{est45}
\end{equation}
for all $n\in \mathbb{N}$.

Moreover, by comparison in \eqref{FaGa1} and \eqref{FaGa2}, we can easily deduce also the estimates for the
time derivatives
$\partial_t\xivec_n$ and $\partial_t\eta_n$. Indeed, we have, \red{for every $t\in (0,T)$,}
\begin{align}
&
\Vert\partial_t \xivec_n\Vert_{L^2(0,t;V_{div}^\prime)}\,\leq\,C\,\Vert\hvec\Vert_\mathcal{V},
\,\quad\Vert\partial_t\eta_n\Vert_{L^2(0,t;V')}\,\leq\,C\,\Vert\hvec\Vert_\mathcal{V},
\quad\mbox{ for all }\,n\in \mathbb{N}.
\label{est46}
\end{align}

From \eqref{est45}, \eqref{est46}, we deduce the existence of \red{a subsequence, which is again indexed by $n$,
such that, with two functions $\xivec$, $\eta$ satisfying \eqref{reglin}, we have}
\begin{align*}
&\red{\xivec_n \to \xivec \quad\mbox{weakly$^*$ in }\,H^1(0,T;V_{div}')\cap L^\infty(0,T;G_{div})\cap L^2(0,T;V_{div})\,,}\\
&\red{\eta_n \,\to \eta\quad\mbox{weakly$^*$ in }\,H^1(0,T;V')\cap L^\infty(0,T;H)\cap L^2(0,T;V)\,.}
\end{align*}
Then, by means of standard arguments, we can pass to the limit as $n\to\infty$ in \eqref{FaGa1}--\eqref{FaGa3} and prove
that $\xivec$, $\eta$ satisfy the weak formulation of the problem \eqref{linsy1}--\eqref{linics}.
Notice that we actually have the regularity \eqref{reglin}, since the space $H^1(0,T;V_{div}^\prime)
\cap L^2(0,T;V_{div})$ is continuously embedded in $C^0([0,T];G_{div})$; similarly
we obtain that $\eta\in C^0([0,T];H)$. \red{Moreover, from \eqref{est45}, \eqref{est46} and
the weak and weak$^*$ sequential semicontinuity of norms it follows that \eqref{stabulin} is satisfied.}

Finally, in order to prove that the solution $\xivec,\eta$ is unique, we can
test the difference between \eqref{linsy1}, \eqref{linsy2}, written for two solutions
$[\xivec_1,\eta_1]$ and $[\xivec_2,\eta_2]$, by $\xivec:=\xivec_1-\xivec_2$ and by $\eta:=\eta_1-\eta_2$,
respectively. Since the problem is linear, the argument is straightforward, and the details
can be left to the reader. \end{proof}

\vspace{3mm}
\noindent
\textbf{Differentiability of the control-to-state operator.}
\blue{In this subsection, we} prove the following result which is crucial to establish optimality conditions.

\begin{thm}
\label{diffcontstat}
Let the assumptions of Lemma \ref{stablem2} hold true. Then the control-to-state operator
${\cal S}:{\cal V}\to {\cal H}$ is Fr\'echet differentiable on ${\cal V}$ when viewed as a mapping between the spaces  ${\cal V}$ and ${\cal Z}$, where
\begin{align}
&\mathcal{Z}:=\big[C^0([0,T];G_{div})\cap L^2(0,T;V_{div})\big]\times\big[C^0([0,T];H)\cap L^2(0,T;V)\big].\nonumber
\end{align}
Moreover, for any $\overline{\vvec}\in {\cal V}$, the Fr\'echet derivative $\,{\cal S}'(\overline{\vvec})
\in {\cal L}({\cal V},{\cal Z})\,$ is given by
$\,{\cal S}'(\overline{\vvec})\hvec=[\xivec^{\hvec},\eta^{\hvec}],$
for all $\,\hvec\in\mathcal{V}$,
where $[\xivec^{\hvec},\eta^{\hvec}]$ is the unique weak solution to the linearized system \eqref{linsy1}--\eqref{linics}
at $[\overline{\uvec},\overline{\varphi}]={\cal S}(\overline{\vvec})$ that corresponds to $\hvec\in\mathcal{V}$.
\end{thm}

\begin{proof}
Let $\overline{\vvec}\in {\cal V}$ be fixed and $[\overline{\uvec},\overline{\varphi}]={\cal S}(\overline{\vvec})$.
\red{Recalling \eqref{stabulin}, we first note that the linear mapping $\hvec\mapsto[\xivec^{\hvec},\eta^{\hvec}]$
belongs to $\mathcal{L}(\mathcal{V},\mathcal{Z})$, in particular.}
Moreover, let $\Lambda>0$ be fixed. In the following, we consider perturbations $\hvec \in {\cal V}$ such that
$\,\|\hvec\|_{\cal V}\,\le\,\Lambda$. For any such perturbation $\hvec$, we put
\begin{align}
&[\uvec^{\hvec},\varphi^{\hvec}]:={\cal S}(\overline{\vvec}+\hvec),\qquad
\pvec^{\hvec}:=\uvec^{\hvec}-\overline{\uvec}-\xivec^{\hvec},\qquad q^{\hvec}:=\varphi^{\hvec}-\overline{\varphi}-\eta^{\hvec}.\nonumber
\end{align}
Notice that we have the regularity
\begin{align}
&
\pvec^{\hvec}\in H^1(0,T;V_{div}')\cap C^0([0,T];G_{div})\cap L^2(0,T;V_{div}),
\nonumber
\\
&
q^{\hvec} \in H^1(0,T;V')\cap C^0([0,T];H)\cap L^2(0,T;V)\,.
\end{align}
By virtue of (\ref{bound1}) and of (\ref{stabest2}), there is a constant $C_1^*>0$,
which may depend on the data of the problem and on $\Lambda$, such that we have:
for every $\hvec\in {\cal V}$ with $\|\hvec\|_{\cal V}\le\Lambda$ it holds
\begin{align}
&
\left\|[\uvec^{\hvec},\varphi^{\hvec}]\right\|_{\cal H}\,\le\,C_1^*\,,\quad\,
\|\varphi^{\hvec}\|_{C^0(\overline{Q})}\,\le\,C_1^*\,,
\label{bound2}\\[2mm]
&
\Vert\uvec^{\hvec}-\overline{\uvec}\Vert_{{C^0([0,t];G_{div})}
\cap L^2(0,t;V_{div})}^2\,+\,\Vert\varphi^{\hvec}-\overline{\varphi}\Vert_{H^1(0,t;H)\cap C^0([0,t];V)
\cap L^2(0,t;H^2(\Omega))}^2
\le \,C_1^*\,\|\hvec\|_{{\cal V}}^2,
\label{bound3}
\end{align}
for every $t\in (0,T]$.

After some straightforward \red{algebraic manipulations}, we can see that
$\pvec^{\hvec}, q^{\hvec}$ (which, for simplicity, shall henceforth be denoted by
$\pvec,q$) is a solution to the weak analogue of the following problem:
\begin{align}
&\pvec_t-2\,\mbox{div}\big(\nu(\overline{\varphi})D\pvec\big)-2\,\mbox{div}\big((\nu(\varphi^{\hvec})
-\nu(\overline{\varphi}))D(\uvec^{\hvec}-\overline{\uvec})\big)
-2\,\mbox{div}\big((\nu(\varphi^{\hvec})-\nu(\overline{\varphi})-\nu^{\,\prime}(\overline{\varphi})\eta)D\overline{\uvec}\big)\nonumber\\
&\quad+\big((\uvec^{\hvec}-\overline{\uvec})\cdot\nabla\big)(\uvec^{\hvec}-\overline{\uvec})
+(\pvec\cdot\nabla)\overline{\uvec}+(\overline{\uvec}\cdot\nabla)\pvec+\nabla\widetilde{\pi}^{\hvec}\nonumber\\
&=-\big(K\ast(\varphi^{\hvec}-\overline{\varphi})\big)\nabla(\varphi^{\hvec}-\overline{\varphi})
-(-K\ast q)\nabla\overline{\varphi}-(K\ast\overline{\varphi})\nabla q\,\quad\mbox{ in }Q,\\
&q_t+(\uvec^{\hvec}-\overline{\uvec})\cdot\nabla(\varphi^{\hvec}-\overline{\varphi})+\pvec\cdot\nabla\overline{\varphi}+\overline{\uvec}\cdot\nabla q
=\mbox{div}\big(\lambda(\overline{\varphi})
\nabla q\big)\nonumber\\
&\quad+\mbox{div}\big((\lambda(\varphi^{\hvec})-\lambda(\overline{\varphi}))\nabla(\varphi^{\hvec}-\overline{\varphi})\big)
+\mbox{div}\big((\lambda(\varphi^{\hvec})-\lambda(\overline{\varphi})-(
\lambda'(\overline{\varphi}))\eta)\nabla\overline{\varphi}\big)\nonumber\\
&\quad-\mbox{div}\big((m(\varphi^{\hvec})-m(\overline{\varphi}))\nabla K\ast(\varphi^{\hvec}-\overline{\varphi})\big)
\nonumber\\
&\quad-\mbox{div}\big((m(\varphi^{\hvec})-m(\overline{\varphi})-m'(\overline{\varphi})\eta)\nabla K\ast\overline{\varphi}\big)
-\mbox{div}\big(m(\overline{\varphi})\nabla K\ast q\big)\,\quad\mbox{ in }Q,\\
&\mbox{div}(\pvec)=0\, \quad\mbox{ in }Q,\\
&\pvec=\mathbf{0}\,\quad\mbox{ on }\Sigma,\\
&\big[\lambda(\overline{\varphi})\nabla q+\big(\lambda(\varphi^{\hvec})-\lambda(\overline{\varphi})\big)\nabla(\varphi^{\hvec}-\overline{\varphi})
+\big(\lambda(\varphi^{\hvec})-\lambda(\overline{\varphi})-(
\lambda'(\overline{\varphi}))\eta\big)\nabla\overline{\varphi}\nonumber\\
&\quad-(m(\varphi^{\hvec})-m(\overline{\varphi}))\nabla K\ast(\varphi^{\hvec}-\overline{\varphi})
-(m(\varphi^{\hvec})-m(\overline{\varphi})-m'(\overline{\varphi})\eta)\nabla K\ast\overline{\varphi}\nonumber\\
&\quad-m(\overline{\varphi})\nabla K\ast q\big]\cdot\nvec=0\,
\quad\mbox{ on }\Sigma,\label{bcondpq}\\
&\pvec(0)=\mathbf{0},\qquad q(0)=0\,\quad\mbox{ in }\Omega.\label{icondpq}
\end{align}
That is, $\pvec$ and $q$ solve the following variational problem (where we avoid to write the
argument $t$ of the involved functions):
\begin{align}
&\langle\pvec_t,\wvec\rangle_{{{V_{div}}}} \,+\,2\,
\big(\nu(\overline{\varphi})D\pvec,D\wvec\big)
\,+\,2\,\big((\nu(\varphi^{\hvec})-\nu(\overline{\varphi}))D(\uvec^{\hvec}-\overline{\uvec}),D\wvec\big)
\nonumber\\[1mm]
&\quad+\,2\,\big((\nu(\varphi^{\hvec})-\nu(\overline{\varphi})-\nu'(\overline{\varphi})\eta^{\hvec})D\overline{\uvec},D\wvec\big)
+b(\pvec,\overline{\uvec},\wvec)+b(\overline{\uvec},\pvec,\wvec)\nonumber\\[1mm]
&\quad
+b(\uvec^{\hvec}-\overline{\uvec},\uvec^{\hvec}-\overline{\uvec},\wvec)\nonumber\\[1mm]
&=\,-\big(\big({K}\ast(\varphi^{\hvec}-\overline{\varphi})\big)\nabla(\varphi^{\hvec}-\overline{\varphi}),\wvec\big)
-\big(({{K}}\ast q)\nabla\overline{\varphi},\wvec\big)-\big((K\ast\overline{\varphi})\nabla q,\wvec\big),
\label{wfpeq}\\
&\langle q_t,\psi\rangle_{{{V}}}
+\big((\uvec^{\hvec}-\overline{\uvec})\cdot\nabla(\varphi^{\hvec}-\overline{\varphi}),\psi\big)
+\big(\pvec\cdot\nabla\overline{\varphi},\psi\big)+\big(\overline{\uvec}\cdot\nabla q,\psi\big)
\nonumber\\[1mm]
&=\,-\big(\lambda(\overline{\varphi})\nabla q,\nabla\psi\big)
-\big((\lambda(\varphi^{\hvec})-\lambda(\overline{\varphi}))\nabla(\varphi^{\hvec}-\overline{\varphi}),\nabla\psi\big)
-\big((\lambda(\varphi^{\hvec})-\lambda(\overline{\varphi})-(
\lambda'(\overline{\varphi}))\eta)\nabla\overline{\varphi},\nabla\psi\big)\nonumber\\
&\quad+\big((m(\varphi^{\hvec})-m(\overline{\varphi}))\nabla K\ast(\varphi^{\hvec}-\overline{\varphi}),\nabla\psi)
+\big((m(\varphi^{\hvec})-m(\overline{\varphi})-m'(\overline{\varphi})\eta)\nabla K\ast\overline{\varphi},\nabla\psi\big)\nonumber\\
&\quad+\big(m(\overline{\varphi})\nabla K\ast q,\nabla\psi\big), \label{wfqeq}
\end{align}
for every $\wvec\in V_{div}$, every $\psi\in V$ and almost every $t\in(0,T)$.

We now choose $\,\wvec=\pvec(t)\in V_{div}\,$ and $\,\psi=q(t)\in V\,$ as test functions
in equations \eqref{wfpeq} and \eqref{wfqeq}, respectively. This gives the identities (omitting the explicit dependence on $t$)
\begin{align}
&\frac{1}{2}\,\frac{d}{dt}\,\Vert\pvec\Vert^2
\,+2\int_\Omega\nu(\overline{\varphi})D\pvec:D\pvec \,dx
\,+2\int_\Omega((\nu(\varphi^{\hvec})-\nu(\overline{\varphi}))\,D(\uvec^{\hvec}-\overline{\uvec}):D\pvec\,dx\nonumber\\[1mm]
&\quad+\,2\int_\Omega\nu'(\overline{\varphi})\,q\, D\overline{\uvec}:D\pvec\,dx
\,+\int_\Omega\nu''(\sigma^{\hvec}_1)\,(\varphi^{\hvec}-\overline{\varphi})^2\,
D\overline{\uvec}:D\pvec\,dx
\,+\int_\Omega(\pvec\cdot\nabla)\overline{\uvec}\cdot\pvec\,dx
\nonumber\\[1mm]
&\quad+\int_\Omega\big((\uvec^{\hvec}-\overline{\uvec})\cdot\nabla\big)(\uvec^{\hvec}-\overline{\uvec})\cdot\pvec\,dx
=-\int_\Omega\big(K\ast(\varphi^{\hvec}-\overline{\varphi})\big)\nabla(\varphi^{\hvec}-\overline{\varphi})\cdot\pvec\,dx
\nonumber\\[1mm]
&\quad\,\,-\int_{\Omega}({{K}}\ast q)\nabla\overline{\varphi}\cdot\pvec
\,dx \,-\int_\Omega({{K}}\ast\overline{\varphi})\nabla q\cdot\pvec\,dx
\,,\label{pid}
\\[2mm]
&\frac{1}{2}\,\frac{d}{dt}\,\Vert q\Vert^2
\,+\int_\Omega\big((\uvec^{\hvec}-\overline{\uvec})\cdot\nabla(\varphi^{\hvec}-\overline{\varphi})\big)\,q\,dx\,
+\int_\Omega(\pvec\cdot\nabla\overline{\varphi})\,q\,dx
\nonumber\\[1mm]
&=\,-\int_\Omega\lambda(\overline{\varphi})|\nabla q|^2\,dx
-\int_{\Omega}(\lambda(\varphi^{\hvec})-\lambda(\overline{\varphi}))\nabla(\varphi^{\hvec}-\overline{\varphi})\cdot\nabla q\,dx
\nonumber\\
&\quad-\int_\Omega(\lambda(\varphi^{\hvec})-\lambda(\overline{\varphi})-(
\lambda'(\overline{\varphi}))\eta)\nabla\overline{\varphi}\cdot\nabla q\,dx\nonumber\\
&\quad+\int_\Omega(m(\varphi^{\hvec})-m(\overline{\varphi}))\left(\nabla K\ast(\varphi^{\hvec}-\overline{\varphi})\right)\cdot\nabla q\,dx\nonumber\\
&\quad+\int_\Omega(m(\varphi^{\hvec})-m(\overline{\varphi})-m'(\overline{\varphi})\eta)\left(\nabla K\ast\overline{\varphi}\right)\cdot\nabla q\,dx
+\int_\Omega m(\overline{\varphi})\left(\nabla K\ast q\right)\cdot\nabla q\,dx. \label{qid}
\end{align}
In \eqref{pid}, we have used Taylor's expansion
\begin{align}
&\nu(\varphi^{\hvec})=\nu(\overline{\varphi})+\nu'(\overline{\varphi})(\varphi^{\hvec}-\overline{\varphi})
+\frac{1}{2}\nu''(\sigma_1^{\hvec})(\varphi^{\hvec}-\overline{\varphi})^2,
\end{align}
where
$$\sigma_1^{\hvec}=\theta_1^{\hvec}\varphi^{\hvec}+(1-\theta_1^{\hvec})\overline{\varphi},
\quad \theta_1^{\hvec}=\theta_1^{\hvec}(x,t)\in (0,1). 
$$
Moreover, in the integration by parts on the right-hand side of \eqref{qid}, we employed
the boundary condition \eqref{bcondpq}, which can be written for $\varphi^{\hvec}$ and for $\overline{\varphi}$,
and  \eqref{linsybc}.

We now estimate all of the terms in \eqref{pid} and in \eqref{qid}. Concerning the ones in \eqref{pid},
these can be estimated exactly as in \cite{FRS}. Hence, we just report these estimates omitting the details.
We denote by $C$ positive constants that may depend on the data of the system,
but not on the choice of $\hvec\in {\cal V}$ with $\|\hvec\|_{\cal V}\le\Lambda$,
while $C_\sigma$ denotes a positive constant that also depends on the quantity
indicated by $\,\sigma$.

Denoting by $I^{(4)}_3,\dots, I^{(4)}_7$ the absolute values of the third to seventh terms on the
left-hand side of \eqref{pid}, and by $I^{(5)}_1,\dots, I^{(5)}_3$ the three terms
on the right-hand side, we have, with constants $\epsilon>0$ and $\epsilon'>0$ that will be fixed later, the following
series of estimates:
\begin{align}
&I^{(4)}_3
\leq\,\epsilon\,\Vert\nabla\pvec\Vert^2\,
+\,C_\epsilon\,
\Vert\nabla(\uvec^{\hvec}-\overline{\uvec})\Vert\,\big(\Vert\uvec^{\hvec}\Vert_{H^2(\Omega)^2}+\Vert\overline{\uvec}\Vert_{H^2(\Omega)^2}\big)
\,\|\hvec\|_{\cal V}^2\,,\label{est37}
\\
&
I^{(4)}_4
\leq\,\epsilon\,\Vert\nabla\pvec\Vert^2\,+\,\epsilon'\,\Vert\nabla q\Vert^2
\,+\,C_{\epsilon,\epsilon'}\,\big(1+\Vert\overline{\uvec}\Vert_{H^2(\Omega)^2}^2\big)\,\Vert q\Vert^2\,,
\\
&
I^{(4)}_5
\leq\,\epsilon\,\Vert\nabla\pvec\Vert^2
\,+\, C_\epsilon\,\Vert\overline{\uvec}\Vert_{H^2(\Omega)^2}^2\,
\|\hvec\|_{\cal V}^4\,,
\\
&
I^{(4)}_6
\,\leq\,\epsilon\,\Vert\nabla\pvec\Vert^2\,+\,C_\epsilon\,
\Vert\overline{\uvec}\Vert_{H^2(\Omega)^2}^2\,\Vert\pvec\Vert^2\,,
\\
&
%
I^{(4)}_7
\,\leq\,\epsilon\,\Vert\nabla\pvec\Vert^2\,+\,C_\epsilon\,
\Vert\nabla(\uvec^{\hvec}-\overline{\uvec})\Vert^2\,\|\hvec\|_{\cal V}^2\,,
\\
&
I^{(5)}_1
\,\leq\,\epsilon\,\Vert\nabla\pvec\Vert^2\,+\,C_{\epsilon}\,
\|\hvec\|_{\cal V}^4\,,
\\
&
I^{(5)}_2
\,\leq\,\epsilon\,\Vert\nabla\pvec\Vert^2\,+\,C_{\epsilon}\,\Vert q\Vert^2\,,
\\
&
I^{(5)}_3
\,\leq\,\epsilon'\,\Vert\nabla q\Vert^2\,+\,C_{\epsilon'}\,\Vert\pvec\Vert^2\,.\label{est38}
\end{align}
Let us now consider \eqref{qid}. To estimate some of the terms in this equation, we shall employ the following
identity, which holds for
general functions $G\in C^2([-1,1])$:
\begin{align}
&G(\varphi^{\hvec})-G(\overline{\varphi})-G'(\overline{\varphi})\eta=G'(\overline{\varphi}) q
+\frac{1}{2}G^{''}(\sigma^{\hvec})\left(\varphi^{\hvec}-\overline{\varphi}\right)^2,
\end{align}
with $\sigma^{\hvec}=\theta^{\hvec}\varphi^{\hvec}+(1-\theta^{\hvec})\overline{\varphi}$, $\,\theta^{\hvec}=\theta^{\hvec}(x,t)\in (0,1)$.
\red{We denote by $I^{(6)}_1, I^{(6)}_2$ the absolute values
of the two terms on the left-hand side, which can be estimated exactly as in \cite{FRS} (we
therefore omit the details),
and by $I^{(7)}_1,\dots,I^{(7)}_6$ the six terms on the right-hand side of \eqref{qid}.
Using the mean value theorem, \eqref{GN}, \eqref{bound2}, \eqref{bound3},  H\"older's and Young's inequalities, and the continuity of the embedding
$V\subset L^p(\Omega)$ for $1\le p<+\infty$ in two dimensions of space, we obtain the
following series of estimates:}
\begin{align}
I^{(6)}_1&
\leq\,\epsilon'\,\Vert\nabla q\Vert^2\,+\,\Vert q\Vert^2
\,+\,C_{\epsilon'}\,\Vert\nabla(\uvec^{\hvec}-\overline{\uvec})\Vert^2
\,\Vert\hvec\Vert^2_{\cal V}\,,\label{est39}
\\
I^{(6)}_2&
\leq\,\epsilon\,\Vert\nabla\pvec\Vert^2\,+\,C_\epsilon\,\Vert q\Vert^2\,,\\
I^{(7)}_1&\leq\, -\red{\alpha_0}\,\Vert\nabla q\Vert^2\,,\\
I^{(7)}_2&\leq\,\Vert\lambda(\varphi^{\hvec})-\lambda(\overline{\varphi})\Vert_{L^4(\Omega)}\,
\Vert\nabla(\varphi^{\hvec}-\overline{\varphi})\Vert_{L^4(\Omega)}\,\Vert\nabla q\Vert\nonumber\\
&\leq\,\red{C}\,\Vert\varphi^{\hvec}-\overline{\varphi}\Vert_{L^4(\Omega)}\,
\Vert\nabla(\varphi^{\hvec}-\overline{\varphi})\Vert_{L^4(\Omega)}\Vert\,\Vert\nabla q\Vert\nonumber\\
&\leq\,\epsilon'\,\Vert\nabla q\Vert^2+C_{\epsilon'}\,\Vert\varphi^{\hvec}-\overline{\varphi}\Vert_V^2\,
\Vert\varphi^{\hvec}-\overline{\varphi}\Vert_{H^2(\Omega)^2}^2\,,\\
I^{(7)}_3&\leq\Big(\Vert\lambda'(\overline{\varphi})q\Vert_{L^4(\Omega)}+\frac{1}{2}
\Vert\lambda''(\sigma^{\hvec}_2)(\varphi^{\hvec}-\overline{\varphi})^2\Vert_{L^4(\Omega)}\Big)\,\Vert\nabla\overline{\varphi}\Vert_{L^4(\Omega)}\,
\Vert\nabla q\Vert\nonumber\\
&\leq \,C\big(\Vert q\Vert_{L^4(\Omega)}+\Vert\varphi^{\hvec}-\overline{\varphi}\Vert_{L^8(\Omega)}^2\big)\,\Vert\nabla q\Vert \nonumber\\
&\leq C\,(\Vert q\Vert+\Vert q\Vert^{1/2}\,\Vert\nabla q\Vert^{1/2})\,\Vert\nabla q\Vert
+\Vert\varphi^{\hvec}-\overline{\varphi}\Vert_V^2\,\Vert\nabla q\Vert
\nonumber\\
&\leq \,\epsilon'\,\Vert\nabla q\Vert^2+C_{\epsilon'}\,\Vert q\Vert^2+C_{\epsilon'}\,
\Vert\varphi^{\hvec}-\overline{\varphi}\Vert_V^4
\,\leq\,\epsilon'\,\Vert\nabla q\Vert^2+C_{\epsilon'}\,
\Vert q\Vert^2+C_{\epsilon'}\,\Vert\hvec\Vert_{\mathcal{V}}^4\,,\\
I^{(7)}_4&\leq \,\Vert m(\varphi^{\hvec})-m(\overline{\varphi})\Vert_{L^4(\Omega)}\,
\Vert\nabla K\ast(\varphi^{\hvec}-\overline{\varphi})\Vert_{L^4(\Omega)}\,\Vert\nabla q\Vert
\,\leq\, C\,\Vert\varphi^{\hvec}-\overline{\varphi}\Vert_V^2\,\Vert\nabla q\Vert\nonumber\\
&\leq \,\epsilon'\,\Vert\nabla q\Vert^2+C_{\epsilon'}\,\Vert\varphi^{\hvec}-\overline{\varphi}\Vert_V^4
\,\leq\,\epsilon'\,\Vert\nabla q\Vert^2+C_{\epsilon'}\,\Vert\hvec\Vert_{\mathcal{V}}^4\,,\\
I^{(7)}_5&\leq\, \Big(\Vert m'(\overline{\varphi})\,q\Vert
+\frac{1}{2}\Vert m''(\sigma^{\hvec}_4)(\varphi^{\hvec}-\overline{\varphi})^2\Vert\Big)\,\Vert\nabla K\ast\overline{\varphi}\Vert_{L^\infty(\Omega)}
\,\Vert\nabla q\Vert\nonumber\\
&\leq \,C\big(\Vert q\Vert+\Vert\varphi^{\hvec}-\overline{\varphi}\Vert_{L^4(\Omega)}^2\big)\,\Vert\nabla q\Vert\nonumber\\
&\leq\, \epsilon'\,\Vert\nabla q\Vert^2+C_{\epsilon'}\,\Vert q\Vert^2+C_{\epsilon'}\,
\Vert\varphi^{\hvec}-\overline{\varphi}\Vert_V^4
\,\leq\,\epsilon'\,\Vert\nabla q\Vert^2+C_{\epsilon'}\,
\Vert q\Vert^2+C_{\epsilon'}\,\Vert\hvec\Vert_{\mathcal{V}}^4\,,\\
I^{(7)}_6&\leq\,\Vert m(\overline{\varphi})\Vert_{L^\infty(\Omega)}\,\Vert \nabla K \ast q\Vert\,\Vert\nabla q\Vert
\,\leq\, \epsilon'\,\Vert\nabla q\Vert^2+C_{\epsilon'}\,\Vert q\Vert^2\,.\label{est40}
\end{align}

We now insert estimates \eqref{est37}--\eqref{est38} in \eqref{pid} and the estimates \eqref{est39}--\eqref{est40} in  \eqref{qid}. Adding the resulting inequalities, and taking
$\epsilon,\epsilon'>0$ small enough (in particular, $\epsilon\leq \nu_1/16$ and $\epsilon'\leq \red{\alpha_0}/20$),
we find that
\begin{align*}
&\frac{d}{dt}\big(\Vert\pvec^{\hvec}\Vert^2+\Vert q^{\hvec}\Vert^2\big)
+\nu_1\Vert\nabla\pvec^{\hvec}\Vert^2+\red{\alpha_0}\Vert\nabla q^{\hvec}\Vert^2
\leq\Xi\,(\Vert\pvec^{\hvec}\Vert^2+\Vert q^{\hvec}\Vert^2\big)+\red{\Xi}\,\Vert\hvec\Vert_{\mathcal{V}}^4
+\Xi^{\hvec}\Vert\hvec\Vert_{\mathcal{V}}^2,
\end{align*}
where the functions $\Xi,\Xi^{\hvec}\in L^1(0,T)$ are given by
\begin{align*}
\Xi(t)&:=C\big(1+\Vert\overline{\uvec}(t)\Vert_{H^2(\Omega)^2}^2\big),\\
\Xi^{\hvec}(t)&:= C\Big(\big(\Vert\uvec^{\hvec}(t)\Vert_{H^2(\Omega)^2}+\Vert\overline{\uvec}(t)\Vert_{H^2(\Omega)^2}\big)
\Vert\nabla(\uvec^{\hvec}-\overline{\uvec})(t)\Vert+\Vert\nabla(\uvec^{\hvec}-\overline{\uvec})(t)\Vert^2\nonumber\\
&\quad+\Vert(\varphi^{\hvec}-\overline{\varphi})(t)\Vert_{H^2(\Omega)}^2\Big)\,.
\end{align*}
Recalling that $\|\hvec\|_{\cal V}\le\Lambda$, thanks to (\ref{bound2}) and (\ref{bound3}), we get
\begin{align}
&\int_0^T \Xi^{\hvec}(t)\,dt \,\leq\,C\,\Vert\hvec\Vert_{\mathcal{V}}.\nonumber
\end{align}
Taking \eqref{icondpq} into account, an application of Gronwall's lemma yields the estimate
\begin{align}
&\Vert\pvec^{\hvec}\Vert_{C^0([0,T];G_{div})}^2\,+\,\Vert\pvec^{\hvec}\Vert_{L^2(0,T;V_{div})}^2
\,+\,\Vert q^{\hvec}\Vert_{C^0([0,T];H)}^2
\,+\,\Vert q^{\hvec}\Vert_{L^2(0,T;V)}^2
\,\leq\,C\,\Vert\hvec\Vert_{\mathcal{V}}^3.\nonumber
\end{align}
We therefore have
\begin{align*}
&\frac{\Vert {\cal S}(\overline{\vvec}+\hvec)-{\cal S}(\overline{\vvec})-[\xivec^{\hvec},\eta^{\hvec}]\Vert_{\mathcal{Z}}}
{\Vert\hvec\Vert_{\mathcal{V}}}
\,=\,\frac{\Vert[\pvec^{\hvec},q^{\hvec}]\Vert_{\mathcal{Z}}}{\Vert\hvec\Vert_{\mathcal{V}}}\,
\leq\,C\,\Vert\hvec\Vert_{\mathcal{V}}^{1/2}
\to 0,
\end{align*}
as $\|\hvec\|_{\cal V}\to 0$. This concludes the proof of the assertion.
\end{proof}
\vspace{4mm}\noindent
\textbf{First-order necessary optimality conditions.}
From Theorem \ref{diffcontstat}, by arguing as in the proof of \cite[Corollary 1]{FRS}, we can deduce the following necessary optimality condition:
\begin{cor}
Let the assumptions of Lemma \ref{stablem2} hold true. If $\overline{\vvec}\in\mathcal{V}_{ad}$ is
an optimal control for {\bf (CP)} with associated state $[\overline{\uvec},\overline{\varphi}]={\cal S}(\overline{\vvec})$,
then the following inequality holds true:
\begin{align}
&\beta_1\int_0^T\!\!\int_\Omega(\overline{\uvec}-\uvec_Q)\cdot\xivec^{\hvec}\,dx\,dt
\,+\,\beta_2\int_0^T\!\!\int_\Omega(\overline{\varphi}-\varphi_Q)\,\eta^{\hvec}\,dx\,dt
\,+\,\beta_3\int_\Omega(\overline{\uvec}(T)-\uvec_\Omega)\cdot\xivec^{\hvec}(T)\,dx\nonumber\\
&+\,\beta_4\int_\Omega(\overline{\varphi}(T)-\varphi_\Omega)\,\eta^{\hvec}(T)\,dx
\,+\,\gamma\int_0^T\!\!\int_\Omega\overline{\vvec}\cdot(\vvec-\overline{\vvec})\,dx\,dt\,\geq\, 0\qquad\forall\,\vvec\in\mathcal{V}_{ad},
\label{nec.opt.cond}
\end{align}
where $[\xivec^{\hvec},\eta^{\hvec}]$ is the unique solution to the linearized system \eqref{linsy1}--\eqref{linics}
corresponding to $\hvec=\vvec-\overline{\vvec}$.
\end{cor}

\vspace{4mm}\noindent
\textbf{The adjoint system and first-order necessary optimality conditions.}
\,We now aim to eliminate the variables $\,[\xivec^{\hvec},\eta^{\hvec}]\,$ from the variational
inequality (\ref{nec.opt.cond}). To this end, let
 us introduce the following {\itshape adjoint system}:
\begin{align}
\ptil_t\,=\,&-\,2\,\mbox{div}\big(\nu(\overline{\varphi})\,D\ptil\big)
-(\overline{\uvec}\cdot\nabla)\,\ptil+(\ptil\cdot\nabla^T)\,\overline{\uvec}
\,+\,\qtil\,\nabla\overline{\varphi}-\beta_1(\overline{\uvec}-\uvec_Q),\quad\mbox{ in } Q, \label{adJ1}\\[3mm]
\qtil_t\,=\,&-\mbox{div}\big(\lambda(\overline{\varphi})\nabla\qtil\big)
\,-\, m'(\overline{\varphi})\nabla(K\ast\overline{\varphi})\cdot\nabla\qtil\nonumber\\[1mm]
&\,-\nabla K\dot{\ast}(m(\overline{\varphi})\nabla\qtil)\,+\,\lambda'(\overline{\varphi})\nabla\overline{\varphi}\cdot\nabla\qtil
\,-\,(\nabla K\ast\overline{\varphi})\cdot\ptil\,-\,\nabla K\ast(\overline{\varphi}\,\ptil)\nonumber\\[1mm]
&\,+\,2 \nu^{\,\prime}(\overline{\varphi})\,D\overline{\uvec}:D\ptil\,-\,
\overline{\uvec}\cdot\nabla\qtil
\,-\,\beta_2(\overline{\varphi}-\varphi_Q), \quad\mbox{ in } Q,\label{adJ2}\\[1mm]
\mbox{div}(&\ptil)=0, \quad\mbox{ in } Q,\label{adJ3}\\[1mm]
\ptil=&\mathbf{0},\qquad\frac{\partial\qtil}{\partial\nvec}=0,\quad\mbox{ on }\Sigma,\label{adJ4}\\[1mm]
\ptil(T&)=\beta_3(\overline{\uvec}(T)-\uvec_\Omega),\quad\qtil(T)=\beta_4(\overline{\varphi}(T)-\varphi_\Omega), \quad\mbox{ in } \Omega.
\label{adJics}
\end{align}
Here, we have set
$$(\nabla {{K}}\dot{\ast}\nabla\qtil)(x):=\int_\Omega\nabla {{K}}(x-y)\cdot\nabla\qtil(y) \,dy\, \quad\mbox{for a.\,e. }\,x\in\Omega\,.
$$
Recalling that $\uvec_\Omega\in G_{div}$ and $\varphi_\Omega\in H$, we expect the solution to \eqref{adJ1}--\eqref{adJics}
to have the \red{regularity} properties
\begin{align}
\ptil&\in
H^1(0,T;V_{div}^\prime)
\cap C([0,T];G_{div})\cap L^2(0,T;V_{div}), \label{reg.adJ.sol0}\\
\qtil&\in
H^1(0,T;V')
\cap C([0,T];H)\cap L^2(0,T;V).\label{reg.adJ.sol}
\end{align}
Hence, the pair $[\ptil,\qtil\,]$ must be understood as a solution to the weak formulation
of the system \eqref{adJ1}--\eqref{adJics}. In particular, the following identities must hold:
\begin{align}
&\langle\ptil_t,\zvec\rangle_{{{V_{div}}}}\,=\,2\,\big(\nu(\overline{\varphi})D\ptil,D\zvec\big)
\,-\,b(\overline{\uvec},\ptil,\zvec)+b(\zvec,\overline{\uvec},\ptil)
+\big(\qtil\nabla\overline{\varphi},\zvec\big)-\beta_1\big((\overline{\uvec}-\uvec_Q),\zvec\big),\label{wfadj1}\\[3mm]
&\langle\qtil_t,\chi\rangle_{{{V}}}\,=\,\big(\lambda(\overline{\varphi})\nabla\qtil,\nabla\chi\big)
\,-\, \big(m'(\overline{\varphi})\nabla(K\ast\overline{\varphi})\cdot\nabla\qtil,\chi\big)\nonumber\\[1mm]
&\qquad\quad\quad\quad\,-\big(\nabla K\dot{\ast}(m(\overline{\varphi})\nabla\qtil),\chi\big)\,+\,\big(\lambda'(\overline{\varphi})\nabla\overline{\varphi}\cdot\nabla\qtil,\chi\big)
\,-\,\big((\nabla K\ast\overline{\varphi})\cdot\ptil,\chi\big)\nonumber\\[1mm]
&\qquad\quad\qquad\,-\,\big(\nabla K\ast(\overline{\varphi}\,\ptil),\chi\big)\,+\,2\,\big( \nu^{\,\prime}(\overline{\varphi})\,D\overline{\uvec}:D\ptil,\chi\big)\nonumber\\[1mm]
&\qquad\quad\qquad\,-\,\big(\overline{\uvec}\cdot\nabla\qtil,\chi\big)
\,-\,\big(\beta_2(\overline{\varphi}-\varphi_Q),\chi\big)\,,\label{wfadj2}
\end{align}
for every $\zvec\in V_{div}$, every $\chi\in V$ and almost every $t\in (0,T)$.
\red{We have the following result.}

\begin{prop}
Let the assumptions of Lemma \ref{stablem2} hold true.
Then the adjoint system \eqref{adJ1}--\eqref{adJics} has a unique weak solution $[\ptil,\qtil]$ satisfying
\eqref{reg.adJ.sol0}--\eqref{reg.adJ.sol}.
\end{prop}

\begin{proof}
We only give a sketch of the proof, which can be carried out arguing as the proof of Proposition
\ref{linthm}. In particular, we omit the details of the \red{construction of an
approximating} Faedo--Galerkin scheme and only derive
the basic a priori estimates. To this end,
we take $\zvec=\ptil(t)\in V_{div}$ in \eqref{wfadj1} and $\chi=\qtil(t)\in H$ in (\ref{wfadj2}), and add the resulting
equations.
Omitting the argument $t$ again, we now estimate all the terms on the right-hand side
of the resulting identity. We denote by $C$ positive constants that only depend on the global data and
on $[\overline{\uvec},\overline{\varphi}]$, while $C_\sigma$ stands for positive constants that also
depend on the quantity indicated by the index $\sigma$. Using the elementary Young's inequality \eqref{Young},
the H\"older and Gagliardo--Nirenberg inequalities \red{(cf. \eqref{GN})}, Young's inequality for convolution
integrals, as well as the assumptions and the global bound (\ref{bound1}),
we obtain (with positive constants $\epsilon$ and $\epsilon'$ that will be fixed later) the following series
of estimates:
\begin{align}
&\Big|\int_\Omega(\ptil\cdot\nabla^T)\overline{\uvec}\cdot\ptil\,dx\Big|
\,\leq\,\Vert\ptil\Vert\,\Vert\nabla\overline{\uvec}\Vert_{L^4(\Omega)^{2\times 2}}\,\,\Vert\ptil\Vert_{L^4(\Omega)^2}
\,\leq\,\epsilon\,\Vert\nabla\ptil\Vert^2+C_\epsilon\,\Vert\overline{\uvec}\Vert_{H^2(\Omega)^2}^2\,\Vert\ptil\Vert^2,\\[2mm]
&\Big|\int_\Omega\qtil\,\nabla\overline{\varphi}\cdot\ptil\,dx\Big| \,\leq\,\Vert\qtil\Vert\,\Vert\nabla\overline{\varphi}\Vert_{L^4(\Omega)^2}\,\Vert\ptil\Vert_{L^4(\Omega)^2}
\,\leq\,\epsilon\,\Vert\nabla\ptil\Vert^2+C_\epsilon\,\Vert\qtil\Vert^2,\\[2mm]
&\Big|\beta_1\int_\Omega(\overline{\uvec}-\overline{\uvec}_Q)\cdot\ptil\,dx\Big|
\,\leq\,\beta_1\,\Vert\overline{\uvec}-\overline{\uvec}_Q\Vert\,
\Vert\ptil\Vert\,\leq\,\Vert\ptil\Vert^2
+\frac{\beta_1^2}{4}\,\Vert\overline{\uvec}-\overline{\uvec}_Q\Vert^2,\\[2mm]
&\Big|\int_\Omega m'(\overline{\varphi})\,\qtil\,\nabla(K\ast\overline{\varphi})\cdot\nabla\qtil\,dx\Big|\,\leq\,
m'_\infty\,\Vert\qtil\Vert_{L^4(\Omega)}\,\Vert\nabla(K\ast\overline{\varphi})\Vert_{L^4(\Omega)^2}\,\Vert\nabla\qtil\Vert
\nonumber\\[2mm]
&\,\leq \,C_{m,K}\,\big(\Vert\qtil\Vert+\Vert\qtil\Vert^{1/2}\,\Vert\nabla\qtil\Vert^{1/2}\big)\,\Vert\nabla\qtil\Vert
\,\leq\,\epsilon'\,\Vert\nabla\qtil\Vert^2+C_{\epsilon',m,K}\,\Vert\qtil\Vert^2\,,\\[2mm]
&\Big|\int_\Omega \qtil\,\nabla K\dot{\ast}(m(\overline{\varphi})\nabla\qtil)\,dx\Big|\,\leq\,C_K\,\Vert\qtil\Vert\,
\Vert m(\overline{\varphi})\nabla\qtil\Vert\,\leq\,\epsilon'\,\Vert\nabla\qtil\Vert^2+C_{\epsilon',m,K}\,\Vert\qtil\Vert^2\,,\\[2mm]
&\Big|\int_\Omega\lambda'(\overline{\varphi})\,\qtil\,\nabla\overline{\varphi}\cdot\nabla\qtil\,dx\Big|\,\leq\,
\lambda'_\infty\,\Vert\qtil\Vert_{L^4(\Omega)}\,\Vert\nabla\overline{\varphi}\Vert_{L^4(\Omega)^2}\,\Vert\nabla\qtil\Vert\\
&\,\leq\,
C_{\lambda}\,\big(\Vert\qtil\Vert+\Vert\qtil\Vert^{1/2}\,\Vert\nabla\qtil\Vert^{1/2}\big)\,\Vert\nabla\qtil\Vert
\,\leq\,\epsilon'\,\Vert\nabla\qtil\Vert^2+C_{\epsilon',\lambda}\,\Vert\qtil\Vert^2\,,\\[2mm]
&\Big|\int_\Omega(\nabla K\ast\overline{\varphi})\cdot\ptil\,\qtil\,dx\Big|\,\leq\,C_K(\Vert\ptil\Vert^2+\Vert\qtil\Vert^2)\,,\\[2mm]
&\Big|\int_\Omega\qtil\,\nabla K\ast(\overline{\varphi}\,\ptil)\,dx\Big|\,\leq\,C_K(\Vert\ptil\Vert^2+\Vert\qtil\Vert^2)\,,\\[2mm]
&\Big|2\int_\Omega\big(\nu^{\,\prime}(\overline{\varphi})\,D\overline{\uvec}\!:\!D\ptil\big)\,\qtil
\,dx\Big| \,\leq\,
C_{\nu}\,\Vert D\overline{\uvec}\Vert_{L^4(\Omega)^{2\times 2}}\,\Vert D\ptil\Vert\,
\Vert \qtil\Vert_{L^4(\Omega)}
\\[2mm]
&\quad\le\,
C_{\nu}\,\Vert D\overline{\uvec}\Vert_{L^4(\Omega)^{2\times 2}}\,\Vert D\ptil\Vert\,\big
(\Vert\qtil\Vert+
\Vert\qtil\Vert^{1/2}\,\Vert\nabla\qtil\Vert^{1/2}\big)
\\[2mm]
&\quad\leq\,\epsilon\,\Vert\nabla\ptil\Vert^2+\epsilon'\,\Vert\nabla\qtil\Vert^2
+C_{\epsilon,\epsilon',\nu}\,\big(1+\Vert\overline{\uvec}\Vert_{H^2(\Omega)^2}^2\big)
\,\Vert\qtil\Vert^2,\\[2mm]
&\Big|\beta_2\int_\Omega(\overline{\varphi}-\varphi_Q)\,\qtil\,dx\Big|\,\leq\,\beta_2\,\Vert\overline{\varphi}-\varphi_Q\Vert\,\Vert\qtil\Vert
\,\leq\,\Vert\qtil\Vert^2+\frac{\beta_2^2}{4}\,\Vert\overline{\varphi}-\varphi_Q\Vert^2\,.
\end{align}
Choosing now $\epsilon>0$ and $\epsilon'>0$ small enough (in particular, $\,3\,\epsilon\leq
\nu_1/2\,$ and $\,\red{5\,\epsilon'\leq \alpha_0}/2$), we arrive at the following differential inequality:
\begin{align}
&\frac{d}{dt}\,\big(\Vert\ptil\Vert^2+\Vert\qtil\Vert^2\big)\,+\,\theta_1\,\big(\Vert\ptil\Vert^2+\Vert\qtil\Vert^2\big)+\theta_2\,
\geq\, \nu_1\,\Vert\nabla\ptil\Vert^2+\red{\alpha_0}\,\Vert\nabla\qtil\Vert^2,\label{diffineq3}
\end{align}
where the functions $\,\theta_1,\theta_2\in L^1(0,T)\,$ are given by
\begin{align*}
&\theta_1(t):=C\,\big(1+\Vert\overline{\uvec}(t)\Vert_{H^2(\Omega)^2}^2\big),
\qquad\theta_2(t):=\beta_1^2\,\Vert(\overline{\uvec}-\uvec_Q)(t)\Vert^2+\,\beta_2^2\,\Vert(\overline{\varphi}-\varphi_Q)(t)\Vert^2.
\end{align*}
By applying the (backward) Gronwall lemma to \eqref{diffineq3}, we obtain
\begin{align*}
&\Vert\ptil(t)\Vert^2+\Vert\qtil(t)\Vert^2\leq\Big[\Vert\ptil(T)\Vert^2+\Vert\qtil(T)\Vert^2
+\int_t^T\theta_2(\tau)d\tau\Big]e^{\int_t^T\theta_1(\tau)d\tau}\nonumber\\
&\leq\,C\,\Big[\Vert\ptil(T)\Vert^2+\Vert\qtil(T)\Vert^2+
\beta_1^2\,\Vert\overline{\uvec}-\uvec_Q\Vert_{L^2(0,T;G_{div})}^2
+\beta_2^2\,\Vert\overline{\varphi}-\varphi_Q\Vert_{L^2(Q)}^2\Big],
\end{align*}
for all $t\in[0,T]$. From this estimate, and by integrating \eqref{diffineq3}
over $[t,T]$,
we can
deduce the estimates for $\ptil$ and $\qtil$ in $C^0([0,T];G_{div})\cap L^2(0,T;V_{div})$
and in $C^0([0,T];H)\cap L^2(0,T;V)$, respectively.
A comparison argument in \eqref{adJ1} and \eqref{adJ2} entails the estimates for
$\ptil_t$ and $\qtil_t$ in $L^2(0,T;V_{div}')$ and in $L^2(0,T;V')$, respectively.
We therefore can deduce the existence of a weak solution to system \eqref{adJ1}--\eqref{adJics}
satisfying \eqref{reg.adJ.sol0}--\eqref{reg.adJ.sol}.
The proof of uniqueness is rather straightforward, and \red{we may allow ourselves to leave it
to the interested reader}.
\end{proof}

Using the adjoint system, we can now eliminate $\xivec^{\hvec},\eta^{\hvec}$ from \eqref{nec.opt.cond}.
Indeed, we have the following result.
\begin{thm}
Let the assumptions of Lemma \ref{stablem2} hold true.
If $\overline{\vvec}\in\mathcal{V}_{ad}$ is an optimal control
for {\bf (CP)} with associated state $[\overline{\uvec},\overline{\varphi}]={\cal S}(\overline{\vvec})$
and adjoint state $[\ptil,\qtil]$, then the following variational inequality holds true:
 \begin{align}
&\gamma\int_0^T\!\!\int_\Omega\overline{\vvec}\cdot(\vvec-\overline{\vvec})\,dx\,dt\,+\int_0^T\!\!\int_\Omega\ptil\cdot(\vvec-\overline{\vvec})\,dx\,dt \,\geq\, 0
\,,\quad\forall\,\vvec\in\mathcal{V}_{ad}.\label{nec.opt.cond2}
\end{align}
\end{thm}

\begin{proof}
Note that, thanks to \eqref{adJics}, we have for the sum (that we denote by $\mathcal{I}$)
 of the first four terms on the left-hand side of \eqref{nec.opt.cond} the identity
\begin{align}
&\mathcal{I}:=\beta_1\int_0^T\!\!\int_\Omega(\overline{\uvec}-\uvec_Q)
\cdot\xivec^{\hvec}\,dx\,dt+\beta_2\int_0^T\!\!\int_\Omega(\overline{\varphi}-\varphi_Q)\eta^{\hvec}
\,dx\,dt
+\beta_3\int_\Omega(\overline{\uvec}(T)-\uvec_\Omega)\cdot\xivec^{\hvec}(T)\,dx\nonumber\\
&+\beta_4\int_\Omega(\overline{\varphi}(T)-\varphi_\Omega)\eta^{\hvec}(T)\,dx\,=\,
\beta_1\!\!\int_0^T\int_\Omega(\overline{\uvec}-\uvec_Q)\cdot\xivec^{\hvec}\,dx\,dt\,+\beta_2\int_0^T\!\!\int_\Omega(\overline{\varphi}-\varphi_Q)\eta^{\hvec}\,dx\,dt\nonumber\\
&+\int_0^T\big(\langle\ptil_t(t),\xivec^{\hvec}(t)\rangle_{{{V_{div}}}}\,+
\langle\xivec^{\hvec}_t(t),\ptil(t)\rangle_{{V_{div}}}\big)\,dt
+\int_0^T\big(\langle\qtil_t(t),\eta^{\hvec}(t)\rangle_{{{V}}}
+\langle\eta^{\hvec}_t(t),\qtil(t)\rangle_{{{V}}}\big)
\,dt\,.\label{proofadJ1}
\end{align}
Recalling the weak formulation of the linearized system \eqref{linsy1}--\eqref{linics} for $\hvec=\vvec-
\overline{\vvec}$, we obtain,
omitting the argument $t$,
\begin{align}
\langle\xivec^{\hvec}_t,\ptil\rangle_{{{V_{div}}}}\,&=
\,-2\,\big(\nu(\overline{\varphi})\,D\xivec^{\hvec},
D\ptil\big)\,-\,2\,\big(\nu^{\,\prime}(\overline{\varphi})\,\eta^{\hvec}\,D\overline{\uvec},D\ptil)\,-\,b(\overline{\uvec},
\xivec^{\hvec},\ptil) \nonumber\\[1mm]
&\quad\, -\,b(\xivec^{\hvec},\overline{\uvec},\ptil)
\,+\,\big(\eta^{\hvec}(\nabla K\ast\overline{\varphi}),\ptil\big)\,+\,\big(\overline{\varphi}\,(\nabla K\ast\eta^{\hvec}),\ptil\big)
\,+\,(\vvec-\overline{\vvec},\ptil)\,,
\label{proofadJ2}
\\[3mm]
\langle\eta^{\hvec}_t,\qtil\rangle_{{{V}}}\,&=\,
(\overline{\uvec}\,\eta^{\hvec},\nabla\qtil\,)\,+\,(\xi^{\hvec}\,\overline{\varphi},\nabla\qtil\,)\,
-\,\big(\lambda(\overline{\varphi})\nabla\eta^{\hvec},\nabla\qtil\,\big)\nonumber\\[1mm]
&\quad\,\,+\,\big( m'(\overline{\varphi})\,\eta^{\hvec}\,\nabla(K\ast\overline{\varphi}),\nabla\qtil\,\big)
\,+\,\big(m(\overline{\varphi})(\nabla K\ast\eta^{\hvec}),\nabla\qtil\,\big)\nonumber\\[1mm]
&\quad\,\,-\,\big(\eta^{\hvec}\lambda'(\overline{\varphi})\,\nabla\overline{\varphi},\nabla\qtil\,) .
\label{proofadJ3}
\end{align}
We now insert these two identies, as well as (\ref{wfadj1})  and (\ref{wfadj2}),
in (\ref{proofadJ1}). Integrating by parts, using the boundary conditions for the involved quantities and the fact
that $\xivec^{\hvec}$ and $\ptil$ are divergence free vector fields, and observing that the
symmetry of the kernel $K$ implies the  identity
\begin{align*}
&\int_\Omega ({{K}}\ast\eta)\,\omega\,dx\,=\,\int_\Omega ({{K}}\ast\omega)\,\eta\,dx\,,\quad\forall\,\eta,\omega\in H,
\end{align*}
we \red{arrive at the conclusion that}
$$
\mathcal{I}=\int_0^T\!\!\int_\Omega\ptil\cdot(\vvec-\overline{\vvec})\,dx\,dt\,.
$$
Therefore, (\ref{nec.opt.cond2}) follows from this identity and \eqref{nec.opt.cond}.
\end{proof}

\begin{oss}{\upshape
System \eqref{stt1}--\eqref{stt5} subject to \eqref{sy6}, written for $[\overline{\uvec},\overline{\varphi}]$, the adjoint system \eqref{adJ1}--\eqref{adJics}, and the variational inequality \eqref{nec.opt.cond2}, form together the first-order necessary
optimality conditions. Moreover, since $\mathcal{V}_{ad}$ is a nonempty, closed and convex subset of $L^2(Q)^2$,
the condition \eqref{nec.opt.cond2} is, in the case $\gamma>0$, equivalent to the following condition for the optimal control $\overline{\vvec}\in\mathcal{V}_{ad}$,
\begin{align}\nonumber
&\overline{\vvec}=\mathbb{P}_{\mathcal{V}_{ad}}\Big(-\frac{\ptil}{\gamma}\Big),
\end{align}
where $\mathbb{P}_{\mathcal{V}_{ad}}$ is the orthogonal projector in $L^2(Q)^2$ onto $\mathcal{V}_{ad}$.
From standard arguments it follows from this projection property the pointwise
condition
\begin{align}\nonumber
&
\overline{v}_i(x,t)\,=\,\max\,\left\{v_{a,i}(x,t),\,\min\,\left
\{-\gamma^{-1}\,\widetilde{p}_i(x,t), \,v_{b,i}(x,t)
\right\}\right\}, i=1,2, \quad\mbox{for a.\,e. }\,(x,t)\in Q\,.
\end{align}
}
\end{oss}

\textbf{Acknowledgments}. The first two authors are members of the
Gruppo Nazionale per l'Analisi Matematica, la Probabilit\`{a} e le loro
Applicazioni (GNAMPA) of the Istituto Nazionale di Alta Matematica (INdAM).
The first author is ``titolare di un Assegno di Ricerca dell'Istituto Nazionale di Alta Matematica".


\begin{thebibliography}{99}


\bibitem{AMW} D.\,M. Anderson, G.\,B. McFadden, A.\,A. Wheeler,
    {\itshape Diffuse-interface methods in fluid mechanics},
    Annu. Rev. Fluid Mech. \textbf{30}, Annual Reviews, Palo
    Alto, CA, 1998, 139-165.

\bibitem{AG} S. Agmon, Lectures on elliptic boundary value problems, Mathematical Studies, Van Nostrand, New York, 1965.



\bibitem{BRB} J. Bedrossian, N. Rodr\'{i}guez, A. Bertozzi, {\itshape Local and
global well-posedness for an aggregation equation and Patlak-Keller-Segel
models with degenerate diffusion}, Nonlinearity \textbf{24} (2011),
1683-1714.

\bibitem{BIN} O.\,V. Besov, V.\,P. Il'in, S.\,M. Nikol'ski\u{i}, Integral representations of functions and
embedding theorems. Vol. II, Scripta Series in Mathematics. Edited by M.\,H. Taibleson. V.\,H. Winston \& Sons,
Washington, D.\,C.; Halsted Press [John Wiley \& Sons], New York-Toronto, Ont.-London, 1979.


\bibitem{BELM} S. Bastea, R. Esposito, J.\,L. Lebowitz, R. Marra, {\itshape %
Sharp interface motion of a binary fluid mixture}, J. Stat. Phys. \textbf{124%
} (2006), 445-483.

\bibitem{B} F. Boyer, {\itshape Mathematical study of multi-phase flow under
shear through order parameter formulation}, Asymptot. Anal. \textbf{20}
(1999), 175-212.

\bibitem{CH} J.\,W. Cahn, J.\,E. Hilliard, {\itshape Free energy of a
    nonuniform system. I. Interfacial free energy}, J. Chem. Phys.
    \textbf{28} (1958), 258-267.


\bibitem{CFG} P. Colli, S. Frigeri, M. Grasselli, {\itshape Global existence
of weak solutions to a nonlocal Cahn--Hilliard--Navier--Stokes system}, J.
Math. Anal. Appl. \textbf{386} (2012), 428-444.

\bibitem{CGS1} \red{P. Colli, G. Gilardi, J. Sprekels, {\em Optimal velocity
control of a viscous Cahn--Hilliard system with convection and dynamic boundary
conditions.} arXiv: 1709.02335 [math. AP] (2017), 1-28.}

\bibitem{CGS2} \red{P. Colli, G. Gilardi, J. Sprekels, {\em Optimal velocity
control of a convective Cahn--Hilliard system with double obstacles and dynamic boundary
conditions: a `deep quench' approach.} arXiv: 1709.03892
[math. AP] (2017), 1-30.}


\bibitem{DiB} E. DiBenedetto, Real Analysis, Birkh\"{a}user, Boston, Advanced Text Series, 2002.



\bibitem{EG} C.\,M. Elliott, H. Garcke, {\itshape On the Cahn--Hilliard
equation with degenerate mobility},  SIAM J. Math. Anal. \textbf{27} (1996),
404-423.

\bibitem{FGG} S. Frigeri, C.\,G. Gal, M. Grasselli, {\itshape On nonlocal Cahn--Hilliard--Navier--Stokes systems in two dimensions},
J. Nonlinear Sci. \textbf{26} (2016), 847-893.

\bibitem{FGGS} S. Frigeri, C.\,G. Gal, M. Grasselli, J. Sprekels, {\itshape Two-dimensional nonlocal
Cahn--Hilliard--Navier--Stokes systems with variable viscosity, degenerate mobility and singular potential},
\red{WIAS Preprint Series No. 2309, Berlin 2016. Submitted.}

\bibitem{FG1} S. Frigeri, M. Grasselli, {\itshape Global and trajectories
attractors for a nonlocal Cahn--Hilliard--Navier--Stokes system}, J. Dynam.
Differential Equations \textbf{24} (2012), 827-856.

\bibitem{FG2} S. Frigeri, M. Grasselli, {\itshape Nonlocal
Cahn--Hilliard--Navier--Stokes systems with singular potentials}, Dyn. Partial
Differ. Equ. \textbf{9} (2012), 273-304.

\bibitem{FGK} S. Frigeri, M. Grasselli, P. Krej\v{c}\'{i}, {\itshape Strong
solutions for two-dimensional nonlocal Cahn--Hilliard--Navier--Stokes systems},
J. Differential Equations \textbf{255} (2013), 2597-2614.

\bibitem{FGR} S. Frigeri, M. Grasselli, E. Rocca, {\itshape A diffuse
interface model for two-phase incompressible flows with nonlocal
interactions and nonconstant mobility}, Nonlinearity \textbf{28}
(2015), 1257-1293.

\bibitem{FRS} S. Frigeri, E. Rocca, J. Sprekels, {\itshape Optimal distributed control of a nonlocal
Cahn--Hilliard/Navier--Stokes system in 2D}, SIAM J. Control Optim.  \textbf{54}  (2016),  221-250.

\bibitem{GGG} C.\,G. Gal, A. Giorgini, M. Grasselli, {\itshape The nonlocal
Cahn--Hilliard equation with singular potential: well-posedness, regularity
and strict separation property},  J. Differential Equations {\bf 263} (2017), 5253-5297.






\bibitem{GL0} G. Giacomin, J.\,L. Lebowitz, {\itshape Exact macroscopic
description of phase segregation in model alloys with long range interactions%
}, Phys. Rev. Lett. \textbf{76} (1996), 1094-1097.

\bibitem{GL1} G. Giacomin, J.\,L. Lebowitz, {\itshape Phase segregation
dynamics in particle systems with long range interactions. I. Macroscopic
limits}, J. Statist. Phys. \textbf{87} (1997), 37-61.

\bibitem{GL2} G. Giacomin, J.\,L. Lebowitz, {\itshape Phase segregation
dynamics in particle systems with long range interactions. II. Phase motion}%
, SIAM J. Appl. Math. \textbf{58} (1998), 1707-1729.

\bibitem{GPV} M.\,E. Gurtin, D. Polignone, J. Vi\~{n}als,
    {\itshape Two-phase binary fluids and immiscible fluids described by
    an order parameter}, Math. Models Meth. Appl. Sci. \textbf{6} (1996),
   8-15.

\bibitem{HMR} M. Heida, J. M\'{a}lek, K.\,R. Rajagopal, {\itshape On
    the development and generalizations of Cahn--Hilliard
    equations within a thermodynamic framework}, Z. Angew. Math. Phys.
    \textbf{63} (2012), 145-169.


\bibitem{HHK} \red{M. Hinterm\"{u}ller, M. Hinze, C. Kahle, {\itshape An adaptive finite element
Moreau-Yosida-based solver for a coupled Cahn--Hilliard/Navier--Stokes system},
J. Comput. Phys.  \textbf{235}  (2013), 810-827.}



\bibitem{HHKW}\red{
 M. Hinterm\"uller, M. Hinze, C. Kahle, T. Kiel, {\em A goal-oriented dual-weighted adaptive finite
 element approach for the optimal control of a nonsmooth Cahn--Hilliard--Navier--Stokes system.}
WIAS Preprint No. 2311, Berlin 2016.}

\bibitem{HW3} \red{
M. Hinterm\"uller, T. Kiel, D. Wegner, {\em Optimal control of a semidiscrete
Cahn--Hilliard--Navier--Stokes system with non-matched fluid densities}.
To appear in SIAM J. Control Optim.}

\bibitem{HW}  \red{
M. Hinterm\"uller, D. Wegner, {\em
Distributed optimal control of the Cahn--Hilliard system including the case of a double-obstacle homogeneous free energy density}.
SIAM J. Control Optim. {\bf 50} (2012), 388-418.}

\bibitem{HW1} \red{
M. Hinterm\"uller, D. Wegner, {\em Optimal control of a semidiscrete
Cahn--Hilliard--Navier--Stokes system}.
SIAM J. Control Optim. {\bf 52} (2014), 747-772.}

\bibitem{HW2}  \red{
M.  Hinterm\"uller,  D.  Wegner, {\em Distributed  and  boundary  control  problems  for  the  semidiscrete Cahn--Hilliard/Navier--Stokes system with nonsmooth Ginzburg--Landau energies}.
Topological Optimization and Optimal Transport, Radon Series on Computational and Applied Mathematics {\bf 17} (2017), 40-63.}



\bibitem{Lions} J.-L.~Lions, Quelques M\'{e}thodes de R\'{e}solution des Probl\`emes aux Limites non
 Li\-n\'eai\-res, Dunod Gauthier--Villars, Paris, 1969.






\bibitem{RS} E. Rocca, J. Sprekels, {\itshape Optimal distributed control of a
nonlocal convective Cahn--Hilliard equation by the velocity in three dimensions},
SIAM J. Control Optim.  \textbf{53} (2015),  1654-1680.




\bibitem{T} R. Temam, Navier--Stokes equations and nonlinear
functional analysis, Second edition, CBMS-NSF Reg. Conf. Ser. Appl. Math.
\textbf{66}, SIAM, Philadelphia, PA, 1995.



\end{thebibliography}
\end{document}